\documentclass[a4paper,oneside,11pt]{amsart}
\usepackage{amssymb,amsfonts,amsmath,amsxtra,mathtools,
mathrsfs,placeins,graphicx,verbatim,stmaryrd,hyperref,cite,color,amsthm}
\usepackage[all]{xy}
\usepackage[T1]{fontenc}
\usepackage{lmodern}

\xyoption{line}

\usepackage{cite}
\usepackage{accents}
\usepackage{amsmath,stackengine}

\newtheorem{theorem}{Theorem}[section]
\newtheorem{cor}[theorem]{Corollary}
\newtheorem{lem}[theorem]{Lemma}
\newtheorem{prop}[theorem]{Proposition}

\newtheorem{conj}[theorem]{Conjecture}

\newtheorem*{theoremA}{Theorem~A}
\newtheorem*{theoremB}{Theorem~B}
\newtheorem*{theoremC}{Theorem~C}

\theoremstyle{definition}

\newtheorem{example}[theorem]{Example}
\newtheorem{defi}[theorem]{Definition}
\newtheorem{rem}[theorem]{Remark}

\numberwithin{equation}{section}

\DeclareMathOperator*{\colim}{colim}
\DeclareMathOperator{\Hom}{Hom}
\DeclareMathOperator{\Fun}{Fun}

\DeclareMathOperator{\Ndg}{N_{dg}}
\DeclareMathOperator{\Ndgp}{N^{\prime}_{dg}}
\DeclareMathOperator{\Ncoh}{N_{coh}}

\DeclareMathOperator{\Ob}{Ob}
\DeclareMathOperator{\Mod}{\!-Mod}
\DeclareMathOperator{\Comod}{\!-Comod}

\newcommand{\Ba}{\mathrm{B}}
\newcommand{\PCVect}{\mathsf{PCVect}}
\newcommand{\dgCat}{\mathsf{dgCat}} 
\newcommand{\dga}{\mathsf{dgAlg}}

\newcommand{\dgco}{\mathsf{dgCoa}^{\mathrm{conil}}}
\newcommand{\cuco}{\mathsf{cuCoa}^{\mathrm{conil}}}
\newcommand{\ptdco}{{\mathsf{ptdCoa}}}
\newcommand{\semialg}{\mathsf{SemiAlg}_{\mathrm{sp}}}
\newcommand{\cusemicoa}{\mathsf{spCoa}}

\newcommand{\Sset}{\mathsf{sSet}}
\newcommand{\sSet}{\mathsf{sSet}}
\newcommand{\qCat}{\mathsf{qCat}}
\newcommand{\sCat}{\mathsf{sCat}}

\newcommand{\Ch}{\mathsf{dgVect}}
\newcommand{\op}{\operatorname{op}}

\newcommand{\utilde}[1]{\widetilde{#1}}

\newcommand{\Dco}{\mathscr {D}^{\operatorname{co}}}

\newcommand{\C}{\mathcal C}
\newcommand{\D}{\mathcal D}

\newcommand{\Eta}{\mathrm{H}}

\def\ground{\mathbf{k}}
\def\cat{\mathcal}
\def\C{\mathcal{C}}

\def\MC{\operatorname{MC}}
\def\Sp{\operatorname{Sp}}

\def\id{\operatorname{id}}
\def\fr{\operatorname{fr}}
\def\co{\operatorname{co}}
\calclayout

\thanks{This work was partially supported by the EPSRC grant EP/N015452/1. The first author acknowledges support by the Deutsche Forschungsgemeinschaft under Germany’s Excellence Strategy -- EXC 2121 “Quantum Universe” -- 390833306. The work was completed in part during the second author's visit to MPIM (Bonn) and he acknowledges excellent working conditions in this institution.}

\begin{document}
	\bibliographystyle{../hsiam2}
	\begin{abstract}
		In this paper we establish  Koszul duality between dg categories and a class of curved coalgebras, generalizing the corresponding result for dg algebras and conilpotent curved coalgebras.
		We show that the normalized chain complex functor transforms the Quillen equivalence between quasicategories and simplicial categories into this Koszul duality. This allows us to give a conceptual interpretation of the dg nerve of a dg category and its adjoint. As an application, we prove that the category of representations of a quasicategory $K$ is equivalent to the coderived category of comodules over $C_*(K)$, the chain coalgebra of $K$. A corollary of this is a characterization of the category of constructible dg sheaves on a stratified space as the coderived category of a certain dg coalgebra.	
	\end{abstract}
	\title[Categorical Koszul duality]{Categorical Koszul duality}
	\author{J. Holstein}
	\address{Department of Mathematics\\
		Universit\"at Hamburg\\
		20146 Hamburg\\
		Germany	
	}
	\email{julian.holstein@uni-hamburg.de}
	
	\author{A.~Lazarev}
	\address{Department of Mathematics and Statistics\\
		Lancaster University\\
		Lancaster LA1 4YF\\United Kingdom}
	\email{a.lazarev@lancaster.ac.uk}
	\maketitle
	\tableofcontents
	
\section{Introduction}
Koszul duality is a collection of related phenomena playing an important role in many subfields of algebra and geometry such as deformation theory \cite{Hinich01}, operads and operadic algebras \cite{Ginzburg94}, representation theory \cite{Beilinson96}, to mention just a few prominent examples among very substantial literature on this subject. The present paper is focused on the part of this vast theory that is relevant to differential graded (dg) associative algebras, dg categories and coalgebras. Associative Koszul duality   states that there is an adjunction between the categories of augmented dg algebras and a class of dg coalgebras (called conilpotent) provided by explicit bar and cobar constructions. Both categories possess structures of model categories and this adjunction can be strengthened to a Quillen equivalence between them.

There is also a version of Koszul duality  between not necessarily augmented dg algebras and conilpotent coalgebras supplied with a `differential', that does not necessarily square to zero,  so-called \emph{curved} coalgebras \cite{Positselski11}.

The main result of this paper is a further generalization of this correspondence with dg algebras replaced by \emph{dg categories}.
The Koszul dual of a dg category $\cat D$, given by a bar construction $\Ba\cat D$ is a \emph{pointed curved coalgebra}. Recall that a coalgebra is pointed if its coradical is a direct sum of copies of the ground field.
This sum is indexed by the set of grouplike elements, and the grouplike elements of $\Ba\cat D$ correspond to the objects of the dg category $\cat D$.
A pointed curved coalgebra consists of a curved coalgebra that is pointed, has a splitting of the coradical and satisfies some compatibilities. 
We denote the category of pointed curved coalgebras equipped with a final object by $\ptdco^*$.
We will construct a model structure on $\ptdco^*$ in Proposition \ref{prop:coalgebramodel}.
We denote by $\dgCat^\prime$ the category of small dg categories modified by the assumption that identity morphisms are not equal to zero except in the case of the category with one object and a single zero morphism. Note that $\dgCat^\prime$  inherits the Dwyer-Kan model structure from $\dgCat$ and is equivalent to it as an $\infty$-category.
Here is our first main result (Theorem \ref{thm:koszulquillen} below):
\begin{theoremA}
	There is a Quillen equivalence $\Omega: \ptdco^* \rightleftarrows \dgCat^\prime: B$ between model categories $\dgCat^\prime$ and $\ptdco^*$.
\end{theoremA}
Restricted to the subcategory of dg categories with one object, this reduces to the ordinary dg Koszul duality between dg algebras and curved conilpotent coalgebras.

The key new idea for proving this result is to interpret a dg category $\cat D$ as 
a monoid in dg bicomodules over the coalgebra $\oplus_{\Ob(\cat D)} \ground$.
We believe that this is of independent interest.
Then one can perform a bar construction in bicomodules to obtain a curved comonoid in bicomodules, which is a relative curved coalgebra.
The adjoint is given by the relative cobar construction in bicomodules.

The bicomodule viewpoint allows us to directly adapt most of the arguments from the conilpotent case treated in \cite{Positselski11}.
Many proofs then boil down to the verification of various algebraic identities.
While relatively straightforward, the computations can be be quite cumbersome, and we have chosen to include some detailed arguments to conceptualize and simplify them. To this end we introduce and make extensive use of the so-called \emph{uncurving functor} which associates to a curved (co)algebra an uncurved one in a universal way (see Proposition \ref{prop:uncurving} and Corollary \ref{cor:uncurvingcoalg}). This allows us to substantially streamline various calculations. 

It is likely that our approach permits the translation of most of the results of \cite{Positselski11} to the setting of dg categories and pointed curved coalgebras with only minimal modification. Such a translation would, of course, necessitate a considerable expansion of this paper and has not been undertaken here. Of particular interest is the treatment, only alluded to here, of $A_\infty$ categories and $A_\infty$ functors based on Koszul duality for dg categories.

Next we show that this dg categorical Koszul duality  is closely related to the coherent nerve of simplicial categories and its adjoint. 
Recall from \cite{Lurie11a} that there is a Quillen equivalence $\mathfrak C: \qCat \rightleftarrows \sCat: \Ncoh$, where we write $\qCat$ for simplicial sets with the Joyal model structure and $\sCat$ for simplicial categories with the Dwyer-Kan model structure. 
We show that this equivalence is, in some sense, a Koszul duality in the nonlinear context. 

More precisely, the normalized chain construction induces functors $\utilde C_*: \qCat \to \ptdco^*$ and $G_*: \sCat \to \dgCat$ which transform $\mathfrak C: \qCat \rightleftarrows \sCat: \Ncoh$ into Koszul duality, see Corollary \ref{cor:nervekoszul}.
For categories with one object this was shown by Rivera and Zeinalian \cite{Zeinalia16}.

Note that the functor $\utilde C_*:\qCat\to\ptdco^*$ is  left Quillen and we denote by $F$ its right adjoint. Thus, $F$ associates to a pointed curved coalgebra a simplicial set; its construction is similar to the construction of a simplicial set out of a commutative dg algebra in rational homotopy theory, cf. \cite{BousfieldGugenheim76}. 

We deduce our second main result (Theorem \ref{thm:dgnerveadjoint} below), the following description of Lurie's dg nerve construction:
\begin{theoremB}
	The functor $\Ndg: \dgCat \to \qCat$ constructed in \cite{Lurie11} is equivalent to the functor $\cat D \mapsto FB(\cat D)$. Its left adjoint is given by $K \mapsto \Omega \tilde C_*(K)$. 
\end{theoremB}

In particular, this provides a clear and conceptual construction of the left adjoint to Lurie's dg nerve, generalizing the construction for quasi-categories with a single object in \cite{Zeinalia16}.
The left adjoint has an especially nice form when
$K$ is a Kan complex; in that case $L(K)$ is (quasi-equivalent to) the dg category whose objects are in 1-1 correspondence with connected components of $K$, the endomorphism dg algebra of each object is the chain algebra of the based loop space on the corresponding component and there are no morphisms between different objects. The right adjoint may also be characterized in terms of Maurer-Cartan elements in a certain simplicial convolution algebra.

For simplicity we have stated these results for the case of dg categories over a field. 
Slightly weakened versions of them hold true if we work in dg categories over the integers. Note, however, that the construction of Koszul duality for algebras or dg categories over integers (or indeed, any commutative ring that is not a field)  and the corresponding coalgebras as a Quillen equivalence meets with technical difficulties. For example, the naive cobar-construction of a coalgebra that is not flat over the ground ring, may have the wrong quasi-isomorphism type and so, the notion of a weak equivalence (which is defined through the cobar-construction) for such coalgebras is problematic.

We also generalize Koszul duality for modules. As in the conilpotent case, there are two results of this type, see Theorem \ref{thm:modulecomodule}: 
\begin{enumerate}\item a Quillen equivalence between dg modules over a dg category $\cat D$ and comodules over the coalgebra $B\cat D$ and
	\item a Quillen equivalence between comodules over a pointed curved coalgebra $C$ and dg modules over the dg category $\Omega C$.\end{enumerate}
This leads to a characterization of the functor category from a quasicategory $K$ to the $\infty$-category of chain complexes in terms of $C_*K$-comodules (Theorem \ref{thm:main} below). 	
\begin{theoremC}
	Let $K$ be an $\infty$-category represented as a simplicial set. 
	Then there is an equivalence between $\Fun(K, \Ndg(\Ch))$ and the coderived $\infty$-category of the coalgebra $C_*K$.	
\end{theoremC}
Note that this category of functors may be considered as the derived $\infty$-category of $K$.
As an application of this result, we show that the $\infty$-category of constructible sheaves of dg vector spaces on a stratified space $X$ is equivalent to the coderived category of the chain coalgebra of the simplicial set $\operatorname{Exit}(X)$ of exit paths on $X$, see Proposition \ref{prop:constructiblecomodules}.

\subsection{Outline of the paper}

In Section \ref{sect:background} we introduce two main objects of study: Split dg \emph{semialgebras} and split curved coalgebras. We show that the category of small dg categories is equivalent to a subcategory of split dg semialgebras.

Section \ref{sect:main} is the main part of the paper. We begin by recalling the classical bar and cobar constructions for dg algebras and curved conilpotent coalgebras. 
We then introduce the uncurving functor in Section \ref{sect:uncurving} to simplify some computations. 
In Section \ref{sect:barsemialgebra} we set up our natural generalization of the bar construction to split dg semialgebras, which will contain dg categories as a special case.
We then construct the cobar construction of a split curved coalgebra in Section \ref{sect:barcoalgebra} and show the two functors are adjoint in Section \ref{sect:koszuladjunction}.
In Section \ref{sect:bardgcategory} we specialise to the adjunction between dg categories and pointed curved coalgebras.
We construct a model structure on pointed curved coalgebras in Section \ref{sect:coalgebramodel}
and prove Theorem A, the Quillen equivalence between pointed curved coalgebras and dg categories in Section \ref{sect:quillenequivalence}. In Section \ref{sect:comodulekoszul} we construct Koszul duality on the level of modules, showing the coderived category of comodules over a split curved coalgebra is equivalent to the derived category over its cobar construction.

In the remainder of the paper we consider three applications. In Section \ref{sect:dgnerve} we rewrite the dg nerve of a dg category in terms of the bar construction (Theorem B) and show that the bar cobar adjunction is a linearization of the equivalence of simplicial categories and quasi-categories.

In Section \ref{sect:functorcategories} we use our results to show Theorem C exhibiting certain functor categories as categories of comodules. In Section \ref{sect:stratified} we apply this to exhibit constructible sheaves on a stratified space $X$ as comodules over the chain coalgebra of the exit paths of $X$.

\subsection{Notation and conventions} We work in the category $\Ch$ of differential graded (dg) vector spaces over a field $\ground$; the grading is always a cohomological $\mathbb Z$-grading.   The $n$-fold shift  of a  graded vector space $V$ is defined as  $V[n]^i= V^{i+n}$ while the element in $V[n]$ corresponding to $v\in V$ will be denoted by $s^nv$. 

The category $\Ch$ is symmetric monoidal, and monoids in it are called dg algebras. Similarly the category of comonoids in $\Ch$ are dg coalgebras. The structure theory of ungraded coalgebras generalizes in a straightforward way to the graded case and we will use results and terminology from standard reference books such as \cite{Sweedler68}. A coalgebra is \emph{cosimple} if it has no proper subcoalgebras and \emph{cosemisimple} if it is a sum of its cosimple coalgebras. A \emph{coradical} of a coalgebra $C$, is the maximal cosemisimple subcoalgebra $C_0$ of $C$. A \emph{pointed} (graded) coalgebra is a coalgebra  whose cosimple subcoalgebras are one dimensional; furthermore a pointed (graded) coalgebra whose coradical is $1$-dimensional is called \emph{conilpotent}.  

Given a dg coalgebra $C$, a left \emph{dg $C$-comodule} is a dg vector space $M$ together with a coaction map $M\to C\otimes M$ subject to the usual coassociativity and counit axioms. The notions of a right dg $C$-comodule and of a dg $C$-bicomodule are defined similarly.

We will also work with \emph{pseudocompact} dg vector spaces, or projective limits of finite-dimensional vector spaces; thus a pseudocompact dg vector space $V$ can be written as $V=\varprojlim_\alpha V_{\alpha}$ for a projective system $\{V_\alpha\}$ of finite dimensional dg vector spaces. The grading for pseudocompact vector spaces is likewise cohomological. The category of pseudocompact dg vector spaces $\PCVect$ is equivalent to the opposite category to $\Ch$ with anti-equivalence established by the $\ground$-linear duality functor. The category $\PCVect$ also admits a symmetric monoidal structure, dual to that in $\Ch$; it will be denoted  simply by $\otimes$. Monoids in $\PCVect$ are called pseudocompact algebras; they form a category  opposite to that of dg coalgebras.

\begin{rem}
	The category of graded pseudocompact algebras has an auxiliary function in this paper. 
	 We found the setting of pseudocompact algebras more productive for concrete calculations than that of coalgebras (though admittedly this is in part due to a psychological effect and tastes may differ).
	In particular, various technical results are significantly easier to prove and even some definitions  easier to state in the algebraic context.
	We will thus freely switch to the dualized setting in the course of this paper when appropriate.
\end{rem}

Given a graded pseudocompact algebra $A$, its maximal semisimple quotient will be denoted by $A_0$; it is \emph{pointed} if $A_0$ is a product of copies of $\ground$ (so that $A$ is dual to a pointed coalgebra).

Occasionally we need to consider the tensor product of a pseudocompact dg vector space $V=\varprojlim_\alpha V_{\alpha}$ and a discrete one $U$; in this situation we will always write $V\otimes U$ for $\varprojlim_\alpha V_{\alpha}\otimes U$; such a tensor product is in general neither discrete nor pseudocompact. 

We will also need the notion of a  Maurer-Cartan (MC) element in an algebra $A$; it is  an element $x\in A^1$ such that $dx+x^2=0$; the set of MC elements in $A$ will be denoted by $\MC(A)$.

A dg category is a category enriched over the category of (co)chain complexes for abelian groups; and unless indicated otherwise, we will always assume that dg categories are, in fact, enriched over $\Ch$. Dg categories possess a model category structure \cite{Tabuada05} where weak equivalences are so-called \emph{quasi-equivalences}, a dg-version of the ordinary notion of equivalence of categories.

We will use some of the language and theory of $\infty$-categories.
In general, this term will stand for an $(\infty,1)$-category considered in a model-agnostic way.

We also use two specific models: \emph{quasicategories} (weakly Kan simplicial sets) and \emph{relative categories}.
We will use standard results and terminology of  quasicategories following \cite{Lurie11, Lurie11a}. 
The symbol $\qCat$ will stand for the category of simplicial sets supplied with the Joyal model structure whose fibrant objects are quasicategories. The weak equivalences are categorical equivalences, denoted by the symbol $\simeq$.
The same category with the ordinary Quillen model structure will be denoted by $\Sset$. 
The subcategories of $\qCat$ and $\sSet$ whose objects are \emph{reduced} simplicial sets will be denoted by $\qCat^0$ and $\sSet^0$. 
Given simplicial sets $K$ and $L$ such that $L$ is a quasicategory we denote by $\Fun(K,L)$ the quasicategory of functors from $K$ to $L$.

Relative categories \cite{Barwick12}
are categories equipped with a special class of morphisms, to be thought of as weak equivalences, and their homotopy theory is equivalent to other models of $\infty$-categories.
Given any category $\cat C$ with a class $W$ of morphisms we may consider it as a relative category, and we call this the $\infty$-category obtained by localizing $\cat C$ at $W$. 
In particular any model category gives rise to an $\infty$-category.

\subsection{Acknowledgements}
The authors benefitted from discussions with J. Chuang, B. Keller, L. Positselski and J. Woolf.
 
We thank Manuel Rivera for useful comments on the first version of this paper, and George Raptis for pointing out one of the results we originally anounced was incorrect (see Remark \ref{rem:mistake}). We thank Arne Mertens for pointing out a gap in the proof of Proposition \ref{prop:coalgebramodel} and useful discussions on fixing it.

We also thank the anonymous referee for helpful comments.

\section{Categories and semialgebras}\label{sect:background}

\subsection{Semialgebras}
To prove Koszul duality for dg categories we will redefine the notion of a $\ground$-linear category as a monoid in a certain monoidal category.

We first need some basics of the theory of comodules over coalgebra.
\begin{defi}
	Let $C$ be a graded coalgebra, $M$ is a left dg $C$-comodule and $N$ be a right dg $C$-comodule. Then their \emph{cotensor product} 
	$M\boxempty_C N$ is the equalizer of the two maps given by the left and right coactions 
	of $C$:
	\[
	\xymatrix{M\boxempty_CN\ar[r]& M\otimes N\ar@<-.5ex>[r]\ar@<.5ex>[r]&M\otimes C\otimes N}
	\]
	 \end{defi}
 If $M,N$ are two dg bicomodules over a coalgebra $C$, then $M\boxempty_CN$ is a dg  $C$-bicomodule in a natural way. This endows the category of dg $C$-bicomodules with a (nonsymmetric) monoidal structure. Its unit is $C$, viewed as a bicomodule over itself.
 
 \begin{defi}\label{def:semialgebra}
 A \emph{semialgebra} is a pair $(A, R)$ consisting of a a graded coalgebra $R$ and a monoid $A$ in dg $R$-bicomodules. 
 
 A homomorphism of semialgebras $(A, R) \to (B, S)$ consists of a homomorphism of graded coalgebras $f: R \to S$ and a dg $S$-bicomodule map $A \to B$ (where $A$ becomes a $S$-bicomodule via $f$), which is compatible with the monoid structure, i.e.\ the following diagram commutes:
 	\[ \xymatrix{
		A \boxempty_{R}A \ar[r]\ar[d] & A \boxempty_{S}A \ar[r] & B \boxempty_{S} B\ar[d]\\
		A\ar[rr]& & B
	}\]
Note that in this setup $A$	is not necessarily a monoid in $S$-bicomodules. For technical reasons we also consider the case that $R$ is the zero coalgebra. The only comodule  over $0$ is the zero vector space, so this gives us the semialgebra $(0, 0)$, which is initial among semialgebras.

Semialgebras together with their morphisms form a category that we denote by $\mathsf{SemiAlg}$. We will often simplify the notation $(A,R)$ to $A$ when it does not cause confusion.
\end{defi}
	
\begin{rem}
 Positselski considers in \cite{Posit10}
a more general notion of a semialgebra as a triple $(A, R, B)$ where $B$ is a $\ground$-algebra and $R$ is an \emph{$B$-coring}. 
This more general notion  reduces to ours when $B=\ground$ and so the coring $R$ becomes a $\ground$-coalgebra. We will occasionally need to consider the case $\ground = \mathbb Z$.
Specializing further $R=\ground$, we see that a semialgebra of the form $(A, \ground)$ is nothing but a $\ground$-algebra.
\end{rem}
There is a corresponding notion of a module (or, more precisely, a semimodule) over a semialgebra:
\begin{defi}\label{def:semimodule}
	A \emph{left} semimodule $M$ over a semialgebra $(A,R)$ is a left dg $R$-comodule endowed with a semi-action map
	$A\boxempty_{R}M\to M$ that is unital and associative. Left semimodules over $(A,R)$ clearly form a category that we will denote by $(A,R)\Mod$. 
	Right semimodules and semibimodules are defined similarly.
\end{defi}
\begin{rem}
	Note that any semialgebra $(A,R)$ is naturally a semibimodule over itself.
\end{rem}

We now restrict to the class of semialgebras that will interest us in this paper.
 \begin{defi}\label{def:semialgebrasplit}
 We say a semialgebra $(A, R)$ is \emph{split} if $R$ is cosemisimple and there is a retract $v: A \to R$ of the unit as a map of $R$-bicomodules. The category of split semialgebras will be denoted by $\semialg$.
\end{defi}
 This retract is not assumed to be compatible with the monoid structure. A morphism of split semialgebras is not required to be compatible with the retract. 
 Therefore all choices of a retract on $(A,R)$ lead to isomorphic split semialgebras. 
 Thus we should not think of the retract as meaningful extra data.
 However, it allows us to carry out the bar construction more cleanly later on.

\begin{rem}
 The main reason for the restriction on $R$ is to ensure that cotensor products over $R$ be exact.
There are likely generalizations of many of our results beyond the cosemisimple case but we will not pursue them in this paper.
\end{rem}

\subsection{Split curved coalgebras}
  We will now consider a dual notion to that of a semialgebra.
  This will essentially be a suitable comonoid in bicomodules.
  We first note that any coalgebra homomorphism $C \to R$ turns $C$ into an $R$-bicomodule, so a comonoid in bicomodules is just a relative coalgebra.
  However, in the setting of curved coalgebras (to be recalled shortly) a number of conditions are needed to make this setup precise.
  
  We will restrict ourselves to the case where the subcoalgebra $R$ is the coradical $C_0$ of $C$.
  \begin{defi}
  	A \emph{split coalgebra} is a coalgebra $C$ equipped with a section $\epsilon: C\to C_0$ of the inclusion $C_0 \to C$.
  \end{defi}
  \begin{rem}
  	In many cases (for example when $\ground$ is a perfect field) such a splitting is guaranteed to exist, however (unless $C$ is cocommutative) it need not be canonical, so we always consider it as part of our data.
  \end{rem}
  
  We will now extend this definition to the curved case.
  Recall the following definition from\cite[Section 3.1]{Positselski11}.
 
\begin{defi}\label{def:curvedalg}
	A \emph{curved  algebra} $A=(A,d,h)$ is a graded  algebra supplied with a derivation $d:A\to A$ (a differential) of degree 1 and an element $h\in A^2$ called the \emph{curvature} of $A$, such that $d^2(x)=[h,x]$ and $d(h)=0$ for any $x\in A$.
	
	A \emph{curved morphism} between two curved  algebras $A\to B$ is a pair $(f,b)$ where $f:A\to B$ is a map of graded  algebras of degree zero and $b\in B^1$ so that:
	\begin{enumerate}
		\item $f(d_Ax)=d_Bf(x)+[b,f(x)]$;
		\item $f(h_A)=h_B+d_B(b)+b^2$.
	\end{enumerate}
	Two such morphisms $(f,b)$ and $(g,c)$ are composed as $(g,c) \circ (f,b) =(g\circ f, c+g(b))$.
	In particular, every map
	$(f,b)$ can be decomposed $(f,b)=(\id,b)\circ(f,0)$.
\end{defi}
\begin{rem}
	In this paper we will, in fact, need  \emph{pseudocompact} curved algebras, whose definition is obtained simply by adding the adjective `pseudocompact' to Definition \ref{def:curvedalg}.
\end{rem}  
Dually there are curved coalgebras.

\begin{defi}\label{def:curvedcoa}
A \emph{curved coalgebra} is a coalgebra $C$ equipped with an odd coderivation $d$ and a homogeneous linear function $h: C \to \ground$ of degree 2, called the \emph{curvature}, such that $(C^*, d^*, h^*)$ is a curved pseudo-compact algebra.

A morphism of curved dg coalgebras from $(C, d_{C}, h_{C})$ to $(D, d_{D}, h_{D})$ is given by the data 
$(f,a)$ where $f: C \to D$ is a morphism of graded coalgebras and $a: C \to \ground$ is a linear map of degree 1,
such that $(f^*, a^*)$ is a curved morphisms $D^* \to C^*$.
The composition rule is $(g,b)\circ(f,a) = (g\circ f, b \circ f + a)$.
\end{defi}

To minimize sign issues we will perform most of our computations using curved pseudo-compact algebras. 

  \begin{defi}\label{def:semicoalgebraworking}
	A \emph{split curved coalgebra} is a pair $(C, \epsilon)$ such that
	\begin{itemize}
		\item $C = (C, \Delta_{\ground}, \epsilon_{\ground}, d, h_\ground)$ is a curved coalgebra (over $\ground$) 
		\item the restriction of $d$ to the coradical $C_0 \hookrightarrow C$ is zero,
		\item $\epsilon: C \to C_0$ is a a coalgebra map compatible with the differential $d$, which is left inverse to $i$.
			\end{itemize}
		We will often write simply $C$ for $(C,\epsilon)$ when it does not cause confusion. 
		\end{defi}
It follows from coassociativity that $\Delta_{\ground}$ factors as $C \xrightarrow{\Delta} C \boxempty_{C_0} C \to C \otimes C$. Thus, $\Delta:C\to C \boxempty_{C_0} C$ exhibits $C$ as a comonoid in the $C_0$-bicomodules.
The inclusion $C_0\hookrightarrow C$ provides a coaugmentation of this comonoid.
The differential $d$ is compatible with the $C_0$-bicomodule structure and the comonoid structure given by $\Delta$ and $\epsilon$.
Note also that there is automatically a curvature $h$ with values in $C_0$, obtained by factorizing the curvature $h_{\ground}: C \to \ground$ as $\epsilon \circ (h_{\ground} \otimes \id_{C_0}) \circ \rho_{C}: C \to C \otimes C_0 \to \ground \otimes C_0 \to \ground$, where $\rho_{C}$ is the right coaction.
We define $h$ as $(h_{\ground} \otimes \id_{C_0}) \circ \rho_{C}$.

Dually, a split curved pseudocompact algebra $(A, u)$ consists of a curved pseudocompact algebra $A$ with maximal semisimple quotient $A_0$, a splitting $u: A_0 \to A$ of the quotient map and a structure of $A$ as a monoid in $A_0$-bicomodules, satisfying the dual conditions to Definition \ref{def:semicoalgebraworking}.

\begin{defi}\label{def:curvedmorphism}
	A morphism $(f, a): (C, \epsilon) \to (D, \delta)$ of split curved coalgebras consists of
\begin{itemize}
	\item a morphism $(f,a_\ground)$ of curved coalgebras 
	\item a factorization of $a_\ground$ as the composition $C \xrightarrow{a} D_0 \to \ground$
	\end{itemize}
such that
\begin{itemize}
	\item $\delta \circ f = f \circ \epsilon$,
	\item $f$ and $a$ are $D_0$-bicomodule maps,
\end{itemize}
where the $D_0$-bicomodule structure on $C$ is induced by $f: C_0\to D_0$.

The composition is then defined as  
\[
(g, b)\circ(f,a) = (g\circ f,  b \circ f + g \circ a).
\]

The category of split curved coalgebras will be denoted by $\cusemicoa$.
\end{defi}

Note that the zero coalgebra, equipped with the zero map to itself, gives us an initial object $(0,0)$ in $\cusemicoa$.

Furthermore, if $(C,\epsilon)$ is a split curved coalgebra then we call $(C^*, \epsilon^*)$ a split curved pseudocompact algebra.
A map $(A, u) \to (B, v)$ of split curved pseudocompact algebras is of the form $(f, b)$ with $f: A \to B$ and $b: A_0 \to B$ of degree 1 satisfying all the compatibilities derived from \ref{def:curvedmorphism}.
Composition is given by
$(g,b)\circ(f,a) = (g\circ f, b \circ f + g \circ a).$

Dualization induces a contravariant equivalence of categories between split curved coalgebras and split curved pseudocompact algebras.
Therefore we denote the latter category by $\cusemicoa^{\text {op}}$.
We will liberally use this equivalence as computations are usually much easier to perform for pseudocompact algebras.

We will also consider comodules over split curved coalgebras; these are nothing but comodules over the underlying curved coalgebra, cf. \cite[Section 4.1]{Positselski11}.

\begin{defi}\label{def:comodule}
	Let $C$ be a split curved coalgebra.
A \emph{left $C$-comodule} is a graded $\ground$-module $M$ endowed with a endomorphism $d_M$ of degree $1$ such that
\begin{itemize}
	\item There is a coaction map $M \to C \otimes M$ compatible with $d_C$ and $d_M$,
	\item For all $m \in M$ we have $d^2(m) = h*m$ where $*$ is the  action of the pseudocompact algebra $(C^*)^{\text{op}}$ on $M$  corresponding to the coaction of $C$ on $M$. 
\end{itemize}
\end{defi}

Given two left $C$-comodules $M$ and $N$ we denote by $\Hom_C(M,N)$ the complex of $C$-homomorphisms between them.
It is given by the graded vector space of homomorphisms from $M$ to $N$ which are compatible with differential and coaction. 
The differentials on $M$ and $N$ induce the differential on $\Hom_C(M,N)$ which squares to zero (even if $d_M$ and $d_N$ do not).
Thus left $C$-comodules form a dg category which we denote by $C\Comod$.

\begin{rem}
Note that a comodule over a split curved coalgebra $C$ is a $C_0$-comodule as $C$ is a $C_0$-bicomodule.
\end{rem}

  \subsection{Categories as semialgebras}
  Next we show that (small) dg categories can be understood as semialgebras.
 
 Let $S$ be a set and $\ground[S]$ be its linearization, i.e. the free $\ground$-module with $S$ as a basis. It has a (cosemisimple) coassociative, cocommutative coalgebra structure with $\Delta(s)=s\otimes s$ for any $s \in S$. 
 Then $S$ can be recovered from $\ground[S]$ as the set of grouplike elements, and coalgebra homomorphisms $\ground[S]\to\ground[T]$ correspond to maps of sets $S \to T$.

We now observe that the data of a small dg  category $\C$ over $\ground$ is equivalent to a semialgebra of the form $(V_{\C},\ground[S])$. The set of objects of $\C$ is given by $S$.
A dg functor $\C\to \D$ is a homomorphism of the corresponding semialgebras.

\begin{prop}\label{prop:comodule}
	The category of semialgebras of the form $(V, \ground[S])$ is equivalent to the category of small dg categories.
\end{prop}
\begin{proof}
Given a set $S$ and a dg $\ground[S]$-bicomodule $V$ together with a monoid structure map $V\boxempty_{\ground[S)]}V\to V$, we construct a dg category $\mathcal C$ by setting $\operatorname{Ob}(\mathcal C)=S$. 
Note that for any $s\in S$ the inclusion $\{s\} \subset S$ determines an inclusion of coalgebras $\ground\hookrightarrow \ground[S]$ and thus, the structure of a $\ground[S]$-comodule on $\ground$; we will denote it by $\ground_s$. Then for $s_1,s_2\in \Ob(\C)$ set $\Hom(s_1,s_2):=\ground_{s_1}\boxempty_{\ground[S]}V\boxempty_{\ground[S)]}{\ground_{s_2}}$. The composition $\Hom(s_1,s_2)\otimes\Hom(s_2,s_3) \to \Hom(s_1, s_3)$ is determined by the monoid structure on $V$.

Conversely, given a dg category $\mathcal C$ with a set of objects $\Ob(\C)$, we define \[V_\C:=\bigoplus_{s_1,s_2\in \Ob(\C)}\Hom(s_1,s_2).\] 
The space $V$ has a natural structure of a $\ground[\Ob(\C)]$-bicomodule $V_\C\to \ground[\Ob(\C)]\otimes V_\C\otimes \ground[\Ob(\C)]$ defined on each summand by the composition
\[
\Hom(s_1,s_2)\hookrightarrow V_\C\cong \ground_{s_1}\otimes V_\C\otimes \ground_{s_2}\to  \ground[\Ob(\C)]\otimes V_\C\otimes \ground[\Ob(\C)].
\]
The monoid structure on $V_\C$ is determined by the composition in $\mathcal C$.

The statement about the functors is likewise straightforward.
\end{proof}	
\begin{rem}\label{rem-order}
	The reader should note, that our convention for the composition of maps sends $(f,g)$ to $f \circ g$ rather than to $g \circ f$ (which is the usual convention in category theory).
	This is more natural in the context of bicomodules.
\end{rem}
We note that the semialgebra corresponding to a dg category is split if and only if all its endomorphism spaces contain non-zero morphisms.

Let $\C$ be a dg category represented by a semialgebra $(V_{\C}, \ground[S])$. 
We will call a left $(V_{\C}, \ground[S])$-semimodule simply a $\C$-module.
 		
Recall that given a dg category $\C$, a (dg) $\C$-module is usually defined as a $\ground$-linear functor $\C \to \Ch$. This agrees with our nomenclature:
\begin{prop}
	The data of a dg functor $\C\to \Ch$ is equivalent to that of a left $(V_{\C}, \ground[{\Ob(\C)}])$-semimodule.
\end{prop}
\begin{proof}
Let $\C$ be a dg category and $(V_{\C}, \ground[S])$ be its corresponding semialgebra.	Given a $(V_{\C}, \ground[S])$ -semimodule $M$ with a structure map $V_{\C}\boxempty_{\ground[S]}M\to M$ and $s\in S=\Ob(\C)$ we define a functor $F:\C\to \Ch$ by $F(s)=\ground_s\boxempty_{\ground[S]}M$. For two objects $s,t\in S$ the semi-action map $V_{\C}\boxempty_{\ground[S]}M\to M$ restricts to
\[
(\ground_{t}\boxempty_{\ground[S]}V_{\C}\boxempty_{\ground[S]}\ground_s)\otimes(\ground_s\boxempty_{\ground[S]}M)\to \ground_t\boxempty_{\ground[S]}M
\]
which can be rewritten as $\Hom_{\C}(s,t)\otimes F(s)\to F(t)$ or $\Hom_{\C}(s,t)\to\Hom( F(s),F(t))$. A straightforward inspection shows that this map preserves compositions and identities.

Conversely, given a dg functor $\C\to \Ch$ we define $M:=\bigoplus_{s\in S}F(s)$; then clearly $M$ is naturally a (left) $\ground[S]$-bicomodule and for $s,t\in S$ the homomorphisms $\Hom_{\C}(s,t)\to\Hom( F(s),F(t))$ combine to give a semi-action map $V_{\C}\boxempty_{\ground[S]}M\to M$.
\end{proof}
\begin{rem}
	For an earlier appearance of categories as monoids (but not as bicomodules) see Section 5 of \cite{Lefevre03}. Even earlier linear categories as  bicomodules make an appearence in Aguiar's PhD thesis \cite[Section 2.3]{Aguiar97}, as was pointed out to us by Arne Mertens  after this paper was completed.
\end{rem}
\section{Koszul duality for categories}\label{sect:main}

\subsection{Bar and cobar construction for (co)algebras}
We begin by reviewing the Koszul duality between dg algebras and curved conilpotent coalgebras following Positselski  \cite{Positselski11}.

We denote by $\dga$ the category of dg algebras over $\ground$ and by $\dga/\ground$ the category of augmented dg algebras. 
Next, $\dgco$ is the category of conilpotent dg coalgebras and $\cuco$ is the category of conilpotent curved coalgebras and $\cuco_*$ the category obtained from $\cuco$ by adding a final object $*$. Let $i: \dgco \to \cuco$ be the inclusion functor.

There is a model structure on $\cuco_*$ and the following theorem holds.
\begin{theorem}\label{thm-barcobaralgebra}
The (reduced) cobar and bar construction provides adjunctions
\[\Omega: \cuco_* \rightleftarrows \dga: \Ba\]
and
\[\Omega i: \dgco \rightleftarrows \dga/\ground: \Ba\]
Composing with the adjunction $U: \dga/\ground \rightleftarrows \dga: (-)\oplus \ground$ we obtain another adjunction
\[U\Omega: \dgco \rightleftarrows \dga: \Ba_{nr}\]
between the reduced cobar and the nonreduced bar construction $\Ba_{nr} = \Ba(- \oplus \ground)$. 

Moreover, all of these adjunctions are Quillen if we choose suitable model structures.\qed
\end{theorem}

The construction of the adjunctions can be found in Section 6.10 in \cite{Positselski11}.

A model structure on curved coalgebras is discussed in Section 9.3 of \cite{Positselski11}.
Positselski does not consider $\dga$ with its usual model category structure on the right hand side, but one may modify his construction of the model structure on curved coalgebras in a natural way.

We note that the reduced bar construction is not defined on the dg algebra $0$, which is the final object in dg algebras.
This may be remedied by defining $B(0) = *$ and $\Omega(*)=0$ by hand.
	
Then we may define the model structure on $\cuco$ by defining cofibrations and weak equivalences by their images under the functor $\Omega$.
Compared to Positselski's model structure on $\cuco$ this only changes which maps to the final object are considered cofibrations and weak equivalences.
This issue will be revisited when we generalize to pointed curved coalgebras in Section \ref{sect:coalgebramodel}. 

It follows from the proof of the theorem that morphism spaces in these adjunctions can be expressed as Maurer-Cartan sets. 

\begin{lem}\label{lem:conilpotentadjoint}
Let $A \in \dga/\ground$ and $C \in \dgco$ be given. Write $\overline A$ for the kernel of the augmentation and $\overline C$ for the cokernel of the coaugmentation.
Then we have
\[\Hom_{\dga/\ground}(\Omega(C), A) \cong \MC(\Hom(\overline C, \overline A)) \cong \Hom_{\dgco}(C, \Ba(A)).\]
Let now $A \in \dga$ and $C \in \cuco$. Then we have
\[\Hom_{\dga}(\Omega(C), A) \cong \MC(\Hom(\overline C, A)) \cong \Hom_{\cuco}(C, \Ba(A)).\]
\pushQED{\qed} 
Applying the first equivalences to $A \oplus \ground$ we have 
\[ \MC(\Hom(\overline C, A)) \cong \Hom_{\dgco}(C, \Ba_{nr}(A)) \cong \Hom_{\cuco}(C, \Ba(A)).\qedhere\]
\popQED
\end{lem}
We will abuse notation and identify $\Omega$, $U \Omega$ and $\Omega i$.

We will see in Corollary \ref{cor:uncurvingcoalg} that the inclusion $i: \dgco \to \cuco_*$ has a right adjoint $\Eta$.
It follows immediately that there is a natural isomorphism of functors $\Eta \circ \Ba \cong \Ba_{nr}$, as these are both right adjoint to $ \Omega \circ i$.
We may then summarize the situation in the following diagram which is commutative in the sense that the two paths formed by composing left adjoint functors lead to isomorphic functors and similarly for right adjoint functors.
\[	\xymatrix@R+1pc@C+1pc{
	\dga/\ground \ar@<-.5ex>_{\Ba}[r] \ar@<-.5ex>_U[d]
	&	\dgco\ar@<-.5ex>_\Omega[l] \ar@<-.5ex>_i[d]\\
	\dga\ar@<-.5ex>_{\Ba}[r] \ar@<-.5ex>_{\oplus \ground}[u] \ar^{\Ba_{nr}}[ru]
	&\cuco\ar@<-.5ex>[l]_\Omega\ar@<-.5ex>_\Eta[u]
}
\]

\subsection{Uncurving}\label{sect:uncurving}
 The inclusion of dg coalgebras in curved coalgebras admits a right adjoint, called the `uncurving' functor. 
We will begin by explicitly constructing the uncurving functor  for curved algebras. 
The same construction will apply to curved pseudocompact and split curved pseudocompact algebras and thus, by dualization to (split) curved  coalgebras.

\begin{defi}\label{defi:uncurvedalgebra}
Given a curved algebra $A$ we define the associated (uncurved) dg algebra $\Eta A$ as follows. The underlying graded algebra of $\Eta A$ is $A\langle\eta\rangle$ where $\eta$ is an element of degree 1. The differential $d^{\Eta}$ on $\Eta A$ is defined by the formulas: 
\begin{enumerate}
	\item $d^{\Eta}a=da-[\eta,a], a \in A\subset \Eta A$.
	\item $d^{\Eta}(\eta)=h-\eta^2$.
\end{enumerate}	
\end{defi}
To motivate this definition let us consider first $A\langle\eta\rangle$ with differential $d$ on $A\subset A\langle\eta\rangle$ and such that $h-d\eta+\eta^2=0$. 
It is easy to see that $A\langle\eta\rangle$ is still a curved algebra with the same curvature $h$ as $A$ and that $-\eta$ is an MC element in it, i.e.\ it satisfies the equation $h+d(-\eta)+(-\eta)^2=0$. 
Thus $\Eta A$ is the twisting of $A\langle\eta\rangle$ by $-\eta$ and, as such, is a dg (uncurved) algebra.

\begin{rem}\label{rem:curvedifuncurved}
	In fact, a simple computation shows that the data $(A, d, h)$ (where $A$ is a graded algebra, $d$ a derivation of degree 1 and $h \in A^2$) forms a curved algebra if and only if $\Eta A$ is a dg algebra, i.e. $d^{\Eta}$ squares to 0.
\end{rem}

Given another curved algebra $B$ and a curved map $(f,b):A\to B$ the induced map $f_b:\Eta A\to \Eta A$ is equal to $f:A\to B\subset \Eta B$ when restricted to $A\subset \Eta A$ and $f_b(\eta_A)=b+\eta_B$. 
\begin{lem}\label{lem:curved}
	$f_b$ is a dg map.
\end{lem}
\begin{proof}
	Let $a\in A\subset \Eta A$. Then
	\[
	d^{\Eta B}f_b(a)=d_Bf(A)-[\eta_B, f(a)].
	\]
	Furthermore,
	\[
	f_b(d^{\Eta A}a)=f_b(d_Aa-[\eta_A,a])=f(d_Aa)-[f_b(\eta_A),f(a)]=f(d_Aa)-[b+\eta_A,f(a)]
	\]
	and taking into account condition (1) of Definition \ref{def:curvedalg}, we conclude that $d^{\Eta B}f_b(a)=f_b(d^{\Eta A}a)$.
	
	Similarly we have
	\[
	d^{\Eta B}f_b(\eta_A)=d^{\eta_B}(b+\eta_B)=d_Bb-[\eta_B,b]+h_B-\eta^2_B
	\]
	Furthermore,
	\[
	f_b(d^{\Eta A}(\eta_A))=f_b(h_A-\eta^2_A)=f(h_A)-(b+\eta_B)^2=f(h_A)-b^2-[b,\eta_B]-\eta_B^2
	\]
	and taking into account condition (2) of Definition \ref{def:curvedalg} and the identity $[b,\eta_B]=[\eta_B,b]$, we conclude that $d^{\Eta B}f_b(\eta_A)=f_b(d^{\Eta A}(\eta_A))$.
\end{proof}
If $(A,u)$ is a split curved pseudo-compact algebra the same arguments go through: 
The underlying pseudocompact algebra for $\Eta A$ is $A\langle \langle\eta \rangle\rangle$ obtained from $A$ by freely adjoining a power series generator $\eta$.
The splitting is given by composing $u$ with the inclusion $A \to A\langle \langle \eta \rangle \rangle$.
The differential $d^\Eta$ is defined by the same formula as above and Lemma \ref{lem:curved} holds  with the same proof. 
In the rest of this subsection
we will formulate and prove results for curved algebras but they have obvious versions, with the same proofs, in the split pseudocompact case.

\begin{prop}\label{prop:uncurving}
	The uncurving functor $A\mapsto \Eta A$, $(f,b) \mapsto f_b$ is left adjoint to the inclusion functor from  dg algebras to curved  algebras.                         
\end{prop}
\begin{proof}
	Let $B$ be a dg algebra. Then a dg algebra map $f:\Eta A\to B$ is determined by a map $\tilde{f}:A\to B$ and the image of $\eta$ in $B$ that we denote by $b\in B$. We claim that $(\tilde{f},b)$ is a curved map $A\to B$. Indeed, for $x\in A$ we have 
	$f(d^{\Eta}x)=f(dx-[\eta, x])=df(x)$ or $\tilde{f}(dx)=[b,\tilde{f}(x)]+df(x).$ 
	
	Similarly, $f(d^{\Eta}\eta)=f(h-\eta^2)=f(h)-b^2=d(b)$. Our claim is proved.
\end{proof}	
The following result is immediate by the duality between coalgebras and pseudocompact algebras.
\begin{cor}\label{cor:uncurvingcoalg}
	The inclusion of (split) dg coalgebras into (split) curved  coalgebras has right adjoint. We will denote this functor by the same symbol $\Eta$ as in the algebra case.\qed
\end{cor}
\begin{rem}
	It follows from Proposition \ref{prop:uncurving} that for a curved algebra $A$ there is a canonical curved map $A\to \Eta A$ such that any curved map from $A$ into a dg algebra factors through it. This map has the form $(i_A,\eta)$ where $i_A$ is the natural inclusion $A\subset \Eta A\cong A\langle\eta\rangle$. 
\end{rem}
We see that curved maps $A\to B$ where the target $B$ is \emph{uncurved} are described solely in terms of uncurved maps. We will see that a similar description can be obtained for \emph{arbitrary} curved maps.

\begin{prop}\label{prop:uncurvingmaps}
	The set of curved maps  $(f,b):A\to B$ between curved algebras $A$ and $B$ is naturally identified with the subset of dg maps $f_b:\Eta A\to \Eta B$ such that $f_b(a)=f(a)$ for $a\in A$ and $f_b(\eta_A)=b+\eta_B$.
\end{prop}
\begin{proof}
	Given a curved map $(f,b):A\to B$ we saw that the map $f_b:\Eta A\to \Eta B$ is a dg map. Conversely, if $f_b$ is a dg map then repeating the calculation in the proof of Lemma \ref{lem:curved} in the opposite direction, we find that $(f,b)$ is a curved map $A\to B$ as required.
\end{proof}	

It is also possible to understand MC elements for curved algebras in terms of their uncurving.
Recall that for a curved algebra $(A,d,h)$, an element $a\in A^1$ is MC if it satisfies the equation $h+da+a^2=0$. 

\begin{prop}\label{prop:curvedMC}
	Let $A$ be a curved algebra. Then $a\in A$ is MC if and only if $a+\eta\in \Eta A$ is MC.
\end{prop}
\begin{proof}
	We have $$d^{\eta}(a+\eta)=da-[\eta,a]+h-\eta^2.$$ Similarly 
	$$(a+\eta)^2=a^2+[a,\eta]+\eta^2.$$ Taking into account that $[\eta,a]=[a,\eta]$, we conclude that 
	$d^{\eta}(a+\eta)+(a+\eta)^2=0$ if and only if $h+da+a^2=0$.
\end{proof}

\subsection{The bar construction for semialgebras}\label{sect:barsemialgebra}
We will now generalize the bar-cobar adjunction in order to extend it to dg categories.

The natural generality in which the Koszul adjunction holds is the category of split semialgebras on one side and that of split curved coalgebras on the other, see Definitions \ref{def:semialgebrasplit} and \ref{def:semicoalgebraworking}. We will thus first prove the adjunction in this case before specialising to dg categories and pointed curved coalgebras.

Note that when a `ground' cosemisimple coalgebra $R$ is fixed and finite-dimensional (and so gives rise by dualization to a semisimple finite-dimensional algebra), then this adjunction, in the augmented and noncurved case, was considered in \cite{vandenBergh15}. However, in order for it to be applicable to dg categories, it is essential to work over non-fixed, possibly infinite-dimensional, cosemisimple coalgebras.  

We would like to define a reduced bar construction following Section 6.1 of \cite{Positselski11}.
However, it is technically easier to begin with the non-reduced bar construction.

\begin{defi}\label{def:nrsemibar}
	Let $(A,R)$ be a semialgebra.
	Then the \emph{non-reduced bar construction} on $A$ is the split dg coalgebra
	$$\Ba_{nr}(A) = \Ba_{nr}(A, R) \coloneqq (T_{R} A[1], m_1 + m_2)$$
	with coradical $R$.
	Here $T_{R}$ is formed using the iterated cotensor product over $R$, i.e. 
	\[
	T_{R}V =  R \oplus V \oplus (V \boxempty_{R} V) \oplus (V \boxempty_{R} V \boxempty_{R}V) \oplus \dots,
	\] 
	the splitting is given by projection to the first factor and the differential is induced by multiplication and differential on $A$:
	$m_2$ is defined as $s b_1 \otimes s b_2 \mapsto (-1)^{|b_1|+1}(s b_1.b_2)$
	while $m_1$ is defined by the differential $s a \mapsto - s d_A a$ on $A$.
	
	The coproduct is given by deconcatenation. 
	The inclusion of $R$ is a coaugmentation.	
\end{defi}
The computation verifying that $m_1+m_2$ squares to zero is the same as in the non-relative case. 
The details are spelled out for example in Section A.1 of \cite{Anel13}.

	Dually we can consider $\Ba^*_{nr}(A):=\Hom_{\ground}(\Ba_{nr}(A),\ground)$ as a pseudocompact algebra in bimodules over the pseudocompact algebra $R^{*}$.
	This may be written as $\hat T_{R^{*}}  A^*[-1]$. 
	Here $\hat T$ is the completion of the tensor algebra in bimodules over $R^{*}$ so that  $$\Ba_{nr}^*(A)\cong R^*\times \prod_{n=1}^{\infty}( A^*[-1])^{\hat{\otimes}_{R^*} n}.$$

If now $(A,v)$ is split, with $v: A \to R$ a retract of the unit $u: R \to A$ (as $R$-bicomodules), then $v$ determines an isomorphism $\Ba_{nr}(A) \cong \Eta \Ba_v(A)$ with the uncurving of a certain split curved coalgebra.
This $\Ba_v(A)$ will be the non-reduced bar construction of $A$.

To explain this, we will perform this construction for the duals.
As a preparation, let $(A,u)$ be a split graded pseudocompact algebra; recall that $u:A_0\to A$ is a section of the quotient map from $A$ onto its maximal semisimple quotient.
Consider $A\langle\langle \eta\rangle\rangle$, the  graded pseudocompact algebra obtained
by freely adjoining a power series variable $\eta$ in degree 1. Then the splitting $u$ induces a splitting $R \to A\langle\langle \eta\rangle\rangle\to R$ that we denote, by an abuse of notation, by the same symbol $u$, making $(A\langle\langle \eta\rangle\rangle, u)$ a split pseudocompact algebra.

Suppose we are given a split dg pseudocompact algebra  whose underlying graded is of the form $(B,u) = (A\langle\langle \eta\rangle\rangle, u)$.
Following Definition \ref{defi:uncurvedalgebra} we define
\begin{align*}
d_A(a)&=d_Ba+[\eta, a]\\
h&=d\eta+\eta^2.
\end{align*} 

Let us assume that $h\in A^2$ and that $d_A$ is an endomorphism of $A$ of degree 1. 
Then we have the following result. 
It allows one to recognize dg algebras of the form $\Eta A$ for a curved algebra $A$.
\begin{lem}\label{lem:recognition}
	The endomorphism $d_A$ makes $A$ into a curved pseudocompact algebra with curvature $h$ and $B\cong \Eta A$.
\end{lem}
\begin{proof}
	Consider the twisted algebra $B^\eta = (B, d^\eta)$ with differential $d^\eta = d+ [\eta, -]$.
	This is a curved algebra with the curvature element $h = d\eta + \eta^2$.
	The twisted differential restricts to $d_A$ on $A$.
	By assumption $A$ is a curved subalgebra, completing the proof.	
\end{proof}	

\begin{prop}\label{prop:reducedunreduced}
	Let $(A,R)$ be a split  semialgebra, choose a splitting $v: A\to R$ giving a decomposition 
	$A\cong \bar{A}\oplus R$. Then there is a structure of a curved pseudocompact algebra on $\hat T_{R^*}\bar A^*[-1]$ such that  $\Eta\hat T_{R^*}\bar A^*[-1]\cong \Ba^*_{nr}(A,R)$.  
\end{prop}
\begin{proof}
	Denote the element in $A^*[-1]\cong \bar A^*[-1]\oplus\ground^*[-1]$ dual to $-1\in A$ by $\eta$; thus there is an isomorphism of graded  pseudocompact algebras 
	\[
	\Ba^*_{nr}(A,R)\cong \hat T_{R^*}\bar A^*[-1]\cong  \hat T_{R^*}\bar A^*[-1] \langle \langle \eta\rangle \rangle.
	\] 
	The differential in $\Ba^*_{nr}(A,R)$ is determined by its values on $\bar A^*[1]$ and on $\eta$ and it is easy to see that it has the following form:
	\begin{align*}
	d_{\Ba_{nr}^*A}(t)&=Y(t)-[\eta,t]\\
	d_{\Ba_{nr}^*A}(\eta)&=X-\eta^2
	\end{align*}
	where $t\in\bar{A}^*[-1]$ and  $Y(t),X\in\hat{T}_{R^*}\bar{A}^*[-1]$ (more precisely, $Y(t)$ and $X$ are sums of quadratic monomials in $\bar{A}^*[1]$ but this  is unimportant for the argument). Note that a similar argument was employed in \cite[Proposition 3.8]{Lazarev03} in the context of $A_\infty$ algebras.
	
	 In particular $d_{\Ba_{nr}^*(A)} + [\eta, -]$ preserves $T_{R^*}\bar{A}^*[-1]$ and applying Lemma \ref{lem:recognition} 
	 we see that the twisted differential $d_{\Ba_{nr}}+[\eta,-]$ makes 
	 	$\hat T_{R^*}\bar A^*[-1]$ into a 
	 	curved algebra with curvature element $X$. It is also clearly split.
\end{proof}
 
\begin{defi}\label{def:semibar}
	Given a split semialgebra $(A,R)$ with $v: A \to R$ a retract of the unit map $u: R \to A$ (as $R$-bicomodules)
	we let $\bar A = A/R \cong \ker(v)$. 
	Then the \emph{dual bar construction} on $(A,R)$ is the split curved pseudocompact algebra structure on $T_{R}\bar A^*[-1]$ constructed in Proposition \ref{prop:reducedunreduced}. 
	The splitting is given by the natural inclusion $R^*\to T_{R^*}A^*[-1] $. 
	The dual bar construction will be denoted by $\Ba^*_v(A, R)$ or $\Ba^*_v(A)$.
	
	The dual split curved coalgebra structure on $T_{R}\bar A[1]$ is called the \emph{bar-construction} of $(A,R)$. 
	The splitting is given by the natural projection $T_{R}A[1] \to R$. The bar construction will be denoted by $\Ba_v(A, R)$ or $\Ba_v(A)$.
\end{defi}
\begin{rem}
	One can of course define differential $m_1 + m_2$ and curvature $h_1 + h_2$ of the reduced bar construction directly.
	
	We write $A = \bar A \oplus R$ and decompose differential and product as $d_{A}=d_{\bar A} + d_{R}$ and $\mu_A = \mu_{\bar A} + \mu_R$.
	
	The differential $m_2$ is defined on $\bar A \otimes \bar A$ by $s b_1 \otimes s b_2 \mapsto (-1)^{|b_1|+1}s \mu_{\bar A}(b_1 \otimes b_2)$
	while $m_1$ is defined by the differential $s a \mapsto - s d_{\bar A}a$ on $\bar A$.
	
	The curvature $h: T_RV \to R$ is induced by $\mu_{R}$ and $d_{R}$ via $h_2(s b_1 \otimes s b_2) = (-1)^{|b_1|+1}\mu_R(b_1 \otimes b_2)$ and $h_1(s b) = -d_R(b)$. 
	
	To compare this expression with the above computation note that by 
	Proposition \ref{prop:reducedunreduced} the curvature is determined by the differential of the variable $\eta$  in the non-reduced bar construction. Since $\eta$ is dual to $-1$ this differential is determined by the image of multiplication and internal differential in the ground ring, i.e.\ $\mu_R$ and $d_R$.
	
	The patient reader may check by hand that this defines a curved coalgebra, invariant under choice of $v$ and adjoint to the cobar construction to be constructed below. The point of our more indirect definition is to avoid these tedious computations in favour of the friendlier ones in Proposition \ref{prop:reducedunreduced}.
	
	An example of coalgebraic computations of the curved bar construction in the context of operads is given in Section 3.3 of of \cite{Hirsh12}.
\end{rem}

Another computation simplified by the relation $\Ba_{nr} = \Eta \Ba_v$ is the following.

\begin{lem}
	The bar construction is functorial.
\end{lem}
\begin{proof}
	For morphisms compatible with the retract this is clear. It remains to compare the bar constructions for two different retracts $v, w: A \to R$. 
	The underlying graded pseudocompact algebras are easily identified and we claim that
	$(\id_{\Ba^*(A)}, v^* - w^*)$ is an isomorphism of curved pseudocompact algebras $\Ba^*_{v}(A) \cong \Ba^*_{w}(A)$.
	
	From Definition \ref{def:semibar} it is clear that $\Eta \Ba^*_v(A) \cong \Eta \Ba^*_w(A)$ are isomorphic, the only difference is the indeterminate, which is $v^*$ respectively $w^*$.
	We apply Proposition \ref{prop:uncurvingmaps}: $(\id, v^*-w^*)$ is a curved algebra map if $\id_{v^*-w^*}$ is a dg map.
	But $\id_{v^*-w^*}$ acts as $v^* \mapsto w^* + (v^* - w^*)$ i.e. is the identity map.
\end{proof}
As the isomorphism $\Ba_v(A) \cong \Ba_w(A)$ is canonical, we can drop the subscript $v$ from our notation.

\subsection{The cobar construction of a split curved coalgebra}\label{sect:barcoalgebra}
\begin{defi}\label{def:semicobar}
Given a split curved coalgebra $(C,\epsilon)$ we define its (reduced) cobar construction as follows. 
Let $\bar C$ be the cokernel of the  inclusion $C_0 \to C$. 
The splitting $C = \bar C \oplus C_0$ allows us to decompose coproduct, differential and curvature. We write $\Delta_{\bar C}:  \bar C  \to \bar C \otimes \bar C$ and $d_{\bar C}: \bar C \to \bar C$ and $h_{\bar C}: \bar C \to C_0$ for the induced maps.

Then the \emph{cobar construction} on $C$ is the semialgebra
$$\Omega(C) \coloneqq \Omega(C,\epsilon) \coloneqq (T_{C_0}\bar C[-1], m_0 + m_1 + m_2)$$ 
where $T_{C_0}$ is again the iterated cotensor product over $C_0$,
The differential  is defined on $\bar C$ by $s^{-1} c \mapsto m_0(s^{-1}c) + m_1(s^{-1}c) + m_2(s^{-1}c)$ with $m_0(s^{-1}c) = + s^{-1} h_{\bar C}$, $m_1(s^{-1}c) =  -s^{-1} d_{\bar C} (c)$ and $m_2(s^{-1}c) = - (-1)^{|c^{(1)}|} s^{-1} c^{(1)} \otimes c^{(2)}$. 
Here the Sweedler notation refers to the comultiplication $\Delta_{\bar C}$.

The product is given by concatenation. $\Omega(A)$ is split via the natural map $T_{C_0} \bar C[-1] \to C_0$.

For a split curved pseudocompact algebra $(A,u)$ we will denote by $\Omega (A,u)$ or $\Omega(A)$ the cobar construction on the dual split curved coalgebra $(A^*, u^*)$.
\end{defi}
As $C_0 \subset C$ is a coaugmentation there is no curvature term (we check below that the differential does indeed square to 0).

We note that the cobar construction  makes sense over the zero coalgebra.
In this case $C$, $C_0$ and $\bar C = C/C_0$ are all equal to the zero coalgebra. 
Then the tensor algebra is 0 in every degree (including $T^0 = C_0$) and we have $\Omega(C) = 0$, the zero semialgebra over the zero coalgebra, corresponding to the empty dg category.

We need to check that $\Omega(C)$ is in fact differential graded.

\begin{lem}\label{lem:cobarsquarezero}
	We have $(m_0 + m_1 +  m_2)^2 = 0$.
\end{lem}
\begin{proof}
	We check that the constant, linear, quadratic and cubic term of $(m_0 + m_1 +  m_2)^2$ all equal 0.
	The equation for the cubic term $m_2^2 = 0$ is coassociativity of $C$ and the quadratic equation $[m_1,m_2] = 0$ expresses the fact that $d_C$ is a coderivation.
	These equations do not involve the curvature. Detailed derivations can be found in Section A.2 of \cite{Anel13}. 
	
	To consider the equations involving curvature, we denote by $(A, d, h)$ the pseudocompact dual of $C$.
	We write $\check m_i$ for the dual of $m_i$.
		
	Then constant term $[\check m_0, \check m_1] = 0$ is equivalent to $d(h) = 0$ and the quadratic term $\check m_1^2 + [\check m_0, \check m_2] = 0$ is equivalent to the condition $d^2 x = [h,x]$.
	
	We spell out the check for the latter assertion.
	Fix $x \in A$ and evaluate $[m_0, m_2]$ on $s^{-1} x \in  A[-1]$.
	\begin{align*}
	[m_0, m_2](s^{-1}x) & = 0 + m_2(s^{-1}h \otimes s^{-1}x + (-1)^{|s^{-1}x|} s^{-1}x \otimes s^{-1}h) \\
	& = -(-1)^{|h|} s^{-1}(hx) - (-1)^{|x| + |s^{-1}x|} s^{-1}(xh) \\
	& = -s^{-1}(hx - xh) = -s^{-1}[h,x] 
	\end{align*}
	
	As it is clear that $\check m_1^2 (s^{-1}x) = s^{-1}d^2(x)$ this completes the proof. 
\end{proof}

\begin{lem}\label{lem:cobarfunctor}
	The reduced cobar construction is a functor $\cusemicoa \to \semialg$.
\end{lem}
\begin{proof}
	We will construct a contravariant functor on the opposite category of $\cusemicoa$, writing $\Omega(A)$ for $\Omega(A^*)$ etc.\ whenever $A$ is a split curved pseudocompact algebra.
	
	Let $(f, b): (A,u) \to (B,v)$ be a map of split curved pseudocompact algebras.
	We note that $b: A_0 \to B$ is an $A_0$-bimodule map, so we may consider it as an element of $B$, and thus also as a map $b: B_0 \to B$.
	
	Using this we may factor $(f,b)$ as $(\id_B,b) \circ (f, 0)$, as one does in the absolute case, and deal with the cases $b=0$ and $f = \id$ separately.
	
	We first note that $f^*: B^* \to A^*$ induces a morphism of cobar constructions by extending to the  tensor algebras, and this is the map induced by $(f, 0)$.
	
	Next we describe the morphism of cobar constructions induced by the change of curvature map $(\id, b): (A, d_A, h_A) \to (B, d_B, h_B)$.
	Here $A \cong B$ as split graded pseudocompact algebras (and so $A_0=B_0$) but $d$ and $h$ vary.
	
	The conditions for $(\id,b)$ to be a curved map is simplified to $d_A(x)=d_B(x)+[b,x]=0$ for any $x\in A$ and $h_A=h_B+db+b^2=0$. 
	
	We write $R$ for $A_0 = B_0$ and denote by $V$ the graded $R$-bimodule underlying both $\bar A$ and $\bar B$
	and consider the cobar-construction $T_RV^*[-1]$. 
	
	We define $\Omega(\id,b)$ to be the affine automorphism $\beta$ of $T_RV^*[-1]$ that is defined on any linear generator $v^*\in V^*[-1]$ by $\beta(v^*)=v^*+v^*(s^{-1}b)$ and then  extended uniquely to a (multiplicative) automorphism of $T_RV^*[-1]$.
	To analyse $\beta$ we consider the constant derivation $\xi$ induced by $s^{-1}b$ on $T_RA^*[-1]$, i.e.\ $\xi:v^*\mapsto v(b)$.
	
	If the ground field $\ground$ has characteristic $0$, we have that $\beta = e^\xi$  since the two maps are multiplicative automorphisms of $T_RV^*[-1]$ taking the same value on $V^*[-1]$.
	In general, even though $e^\xi$ has factorials in the denominators, all denominators clear when applying $e^\xi$ to word monomials inside $T_RV^*[-1]$  and then the corresponding map can be linearly extended to the whole of $T_RV^*[-1]$. So we can still formally write $e^\xi$ for the automorphism $\beta$ of $T_RV^*[-1]$. 
	
	It remains to show that $e^\xi$ is compatible with the differentials $m_{A^*}$ and $m_{B^*}$ which we write as $m_{A^*}=m_0^A+m_1^A+m_2^A$  and $m_{B^*}=m_0^B+m_1^B+m_2^B $.
	We need to check $e^\xi m_{B^*}=m_{A^*} e^\xi$ or $e^\xi m_{B^*}e^{-\xi}=m_{A^*}$. 
	Note that 
	$$e^\xi m_{B^*}e^{-\xi}=e^{\operatorname{ad} \xi}m_{B^*}
	=m_{B^*}+[\xi,m_0^B+m_1^B+m_2^B]+\frac{1}{2!}[\xi,[\xi, m_0^B+m_1^B+m_2^B]].$$
	The higher degree components are absent since $[\xi,-]$ lowers the tensor degree by one and $m_{B^*}$ has no terms of degree higher than two.
	
	The degree 2 component of the above formula is $m_2^B$ so we must have $m_2^A=m_2^B$. 
	In other words, the multiplication on $A$ is the same as the multiplication on $B$.
	
	The degree 1 component is $m_1^B+[\xi, m_2^B]$ so we must have $m_1^B+[\xi, m_2^B]=m_1^A$. 
	Dualizing we see that this is true if and only if $d_A = d_B + [b, -]$ since
	\begin{align*}
	[\xi, m_2^B](v^*)(sx) &= v(m_2^*(\xi^* sx)) \\
	& = v(m_2^*(sb \otimes sx + (-1)^{|sb| |sx|}sx \otimes sb)) \\
	& = v(-(-1)^{|b|}s (b.x) - (-1)^{|x|}) s (x.b)) \\
	& = v(s(b.x - (-1)^{|x|} x.b))= v(s[b,x])	
	\end{align*}
	
	Finally the degree zero component is $m_0^B+[\xi,m_1^B]+\frac{1}{2!}[\xi,[\xi, m_2^B]]$ so we must have $m_0^B+[\xi,m_1^B]+\frac{1}{2!}[\xi,[\xi, m_2^B]]=m_0^A$. 
	But this corresponds exactly to $h_B+db+b^2=h_A$ in the formula for the curved map. This follows as we have seen that $[\xi,m_2^B]$ is the derivation of degree 1 of $T_RV^*[-1]$ corresponding to the operator $[b,-]$ in $B$, and then $[\xi,[\xi, m_2^B]]$ corresponds to $[b,b]=2b^2$.
	
	Thus, we define $\Omega$ by sending $(f, b)$ to the map defined by $f^*$ on $B_0^*$
	and by $v^* \to f^*(v^*) + v(s^{-1}b)$ on $\bar B^*[-1]$.
	Note that $v(s^{-1}b): A_0 \to B \to \ground$ is an element in $A_0^* \subset \Omega(A)$.
	
	We check that $\Omega$ is a contravariant functor: 
	Let $(f,b): A \to B$ and $(g, c): B \to C$ be maps of split curved pseudocompact algebras.
	Then on $v^* \in C^*[1]$ we have:
	\begin{align*}
		\Omega(f,b) \circ \Omega(g,c) (v^*) &= \Omega(f,b) (g^*(v^*) + v(c)) \\
		& = (gf)^*(v^*) + f^* g^*v^*(s^{-1}b) + f^* v(c) \\
		& = \Omega(gf, cf + gb)(v^*)
	\end{align*}
	This completes the proof.
\end{proof}

\begin{rem}
	Of course the definition of curved coalgebras and their morphisms is chosen exactly so that Lemmas \ref{lem:cobarsquarezero} and \ref{lem:cobarfunctor} hold.
	In fact, one could define a curved coalgebra to be a triple $(C, d, h)$ such that $d_{\Omega C}$ squares to zero, and a morphism to be a pair $(f,a)$ such that $\Omega(f,a)$ is compatible with the differential.
\end{rem}

If we fix $R = \ground$ then these constructions reduce to the bar and cobar construction for curved conilpotent coalgebras and dg algebras.

\subsection{Koszul adjunction}\label{sect:koszuladjunction}
We can also define MC elements for the convolution algebra $\Hom(C, A)$ as in the conilpotent case.
We will restrict to MC elements that vanish on the coradical of $C$.
\begin{defi}\label{def:semialgebramc} 
Given $C \in \cusemicoa$ and $(A,S) \in \semialg$ we define 
$$\MC(C, A) \coloneqq \amalg_{f_O: C_0 \to S}\MC( \Hom_{S}(\bar C, A)).$$ 
Here we take the coproduct over coalgebra homomorphisms, each such $f_O$ makes  $\bar C$ into $S$-bicomodules, and $\Hom_S$ denotes morphisms in $S$-bicomodules. 
Then $\Hom_{S}(\bar C, A)$ is a curved convolution algebra. 
For $f, g: \bar C \to A$ we define the product $f * g: c \mapsto (-1)^{|g| |c^{(1)}|} f(c^{(1)})g(c^{(2)})$.
Moreover we define $df: c \mapsto d_A(fc) - (-1)^{|f|}(d_Cc)$ and a curvature $h: c \mapsto - f_O \circ h_C(c)$.

The MC condition on $ \Hom_{S}(C, A)$ is the curved MC condition, i.e.\ the MC elements are all $\xi: \bar C \to A$ such that $d\xi + \xi * \xi + h = 0$. 
\end{defi}
\begin{rem}
When $C_0=S=\ground$ we have that $C$ is just an conilpotent curved coalgebra and $S$ is a dg algebra. In that  case MC elements in the convolution (curved) algebra are alternatively called \emph{twisting cochains}. 
This terminology is adopted in \cite{Positselski11}.
\end{rem}
\begin{theorem}\label{thm-semibarcobar}
There is an adjunction
$$\Omega: \cusemicoa \rightleftarrows \semialg: \Ba.$$
In fact, for any $C \in \cusemicoa$ and $A \in \semialg$ we may write:
$$\Hom_{\cusemicoa}(C, \Ba A) \cong \MC(C, A) \cong \Hom_{\semialg}(\Omega C, A)$$
\end{theorem}
\begin{proof}
The proof follows the conilpotent case.
Fix $C \in \cusemicoa$ and $(A, S)\in \semialg$.

We begin with the equivalence on the right.
We may decompose the morphism space of semialgebras as a disjoint union over coalgebra homomorphisms.  
We may write the cobar construction as $(\Omega C, C_0)$. 
We have
\[
\Hom((\Omega C, C_0), (A, S)) = \amalg_{f_O: \, C_0 \to S} \Hom_{S}(\Omega C, A)
\]
where we take Hom in dg algebras which are $S$-bicomodules. This makes sense as $f_O$ makes $\Omega C$ into a $S$-bicomodule. 
This matches with the decomposition in the definition of $\MC( C, A)$, so we may compare each component.

Now consider a map from $T_{C_0} \bar C[-1]$ to $A$.
As we have fixed $f_O: C_0 \to S$ this factors uniquely
through the free monoid
in $S$-bicomodules $T_S\bar C[-1]$, and it follows that a map $\Omega C \to A$ is determined by a degree 1 map $\bar C \to A$.
Now the condition that $f$ commutes with differentials is
\[
d_A \circ f = f \circ (m_0 + m_1 + m_2),
\]
which we may spell out as
\[
d_A \circ \phi =  f_O \circ h_C - \phi \circ d_C - (-1)^{|c(1)|}\phi(c^{(1)}).\phi(c^{(2)}),
\]
which is in turn equivalent to the MC equation
$d \phi + \phi * \phi + h = 0$ where $h = - f_O \circ h_C$.

The equivalence on the left is technically a little more complicated.

Again we decompose the hom space of split curved coalgebras as a disjoint union over coalgebra homomorphisms.
We have and we have
\[
\Hom(C, \Ba A) = \amalg_{g_O: \, C_0 \to S} \Hom_{S}(C, \Ba A)
\]
where we take Hom in $S$-coalgebras, i.e.\ curved coalgebras which are $S$-bicomodules. 
This matches with the decomposition in the definition of $\MC(C, A)$, so we may compare each component.

Now $T_{S} \bar A[1]$ is cofree as a conilpotent coalgebra in $S$-bicomodules by the lemma in Section 1.2 of \cite{Vanoystaeyen04}. 
Explicitly, a coalgebra map $f: C \to T_{S} \bar A[1]$ 
is given by a coalgebra map $C \to S$ and an $S$-bicomodule map $C \to  \bar A[1]$. 

The first map is just the composition of the counit $C \to C_0$ with $f$. 

The second map is equivalent to a map $\phi: \bar C \to \bar A$ of degree $1$. To see this we note that by assumption $C_0$ is cosemisimple. Thus by compatibility with the coproduct $C_0$ is mapped to $S \subset T_S\bar A[1]$ and the map $C \to \bar A[1]$ is induced by $\phi$.  

It remains to unravel that curved coalgebra maps are exactly those such that $\phi$ satisfies the MC condition. 
One may just check this by hand, but in particular the signs are rather unpleasant.

To simplify this computation we uncurve the (nonunital) curved algebra structure on $\Hom_S(\bar C, A)$.
We dualize and consider $(f,c) \in \Hom(\Ba^*A, C^*)$, which by Proposition \ref{prop:uncurvingmaps} and Definition \ref{def:semibar} is the same as a map $f_c: \Ba^*_{nr}(A) \to \Eta C^*$ that send $a$ to $f(a)$ and $\eta_{\Ba^*A}$ to $\eta_{C^*} + c$.

As $\Ba^*_{nr}(A)$ is free on $A$ the map $f_c$ is represented by an element $\phi: A^* \to \Eta\bar C^*$.
Here we use that the kernel of the augmentation map on $
\Eta C^*$ is $\Eta\bar C^*$.
The standard computation for the bar construction of an augmented algebra shows that $f_c$ is dg if and only if $\phi$ is an MC element in $\Hom(A^*, \Eta\bar C^*) \cong A \otimes \Eta\bar C^*$. See for example Section A.3 in \cite{Anel13}; the computation is unaffected by working in $S$-bicomodules.

Moreover, the MC element $\phi$ takes the form $\phi' + \id_A \otimes \eta_{C^*} \in A \otimes \Eta\bar C^*$.
But  $A \otimes \Eta C^* \cong \Eta(A \otimes C^*)$ and $\eta_{A \otimes C^*}$ is $\id_A \otimes \eta_{C^*}$.
Thus by Lemma \ref{prop:curvedMC} it is identified with an MC element in $A\otimes C$.
\end{proof}

\subsection{The bar construction of a dg category}\label{sect:bardgcategory}

Given a dg category $\cat D$ we consider it as a semialgebra $(V_{\cat D}, \ground[\Ob(\cat D)])$ as in Proposition \ref{prop:comodule}. We can then define the \emph{bar construction of $\cat D$}, denoted $\Ba(\cat D)$ as the split curved coalgebra $B(V_{\cat D}, \ground[\Ob(\cat D)])$.

Here coradical of $\Ba(\cat D)$ is $\ground[\Ob(\cat D)]$, and thus the bar construction is an example of a \emph{pointed curved coalgebra}.
\begin{example}\label{quiver}
Consider a graded quiver $Q$ and associate a dg category $\cat D_Q$ to it as follows. 	The collection of objects of $\cat D$ coincides with the vertices of $Q$ and morphisms between objects $O_1$ and $O_2$ is the graded vector space spanned by the set of corresponding arrows together with identity morphisms if $O_1=O_2$. All compositions not involving the identity morphisms are zero, as well as the differentials. Then $\Ba(\cat D)$ is called the \emph{path coalgebra} of the quiver $Q$; note that it has the zero differential. Note that in the case when $Q$ has only one vertex, $\Ba(\cat D)$ is just the cofree (conilpotent) coalgebra on the set of arrows of $Q$.  In general, $\Ba(\cat D)$ it is known to be hereditary, i.e. the abelian category of its comodules has homological dimension not greater than 1, cf. \cite[Proposition 8.13]{Simson01} or \cite[Theorem 4]{Chin02}.
\end{example}
\begin{defi}
A \emph{pointed curved coalgebra} is a split curved coalgebra  whose coradical is a direct sum of copies of $\ground$. Its dual pseudocompact algebra will be referred to as a
 pointed curved pseudocompact algebra.
\end{defi}

By analogy with categories we denote by $\Ob(C)$ the set of grouplike elements of a pointed dg coalgebra. 
Then the coradical of $C$ is $\ground[\Ob(C)]$. 

We note that any morphism $C \to C'$ of pointed dg coalgebras induces a morphism $\Ob(C) \to \Ob(C')$.

We denote the full subcategory  of $\cusemicoa$ formed by pointed curved coalgebras by $\ptdco$.
We then have the following proposition, which is the  bar-cobar adjunction for dg categories. 

\begin{prop}\label{prop:catadjunction}
There is an adjunction $\Omega:\ptdco\rightleftarrows \dgCat: \Ba$ induced by equivalences $\Hom(\Omega C, \cat D) \cong \MC(C, \cat D) \cong \Hom(C, \Ba\cat D)$. 
\end{prop}
\begin{proof}
We restrict the bar-cobar adjunction from Theorem \ref{thm-semibarcobar} to the subcategories $\dgCat \subset \semialg$ and $\ptdco \subset \cusemicoa$.

Our discussion above shows that the functor $\Ba$ lands in pointed curved coalgebras. Conversely $\Omega(C)$ is a semialgebra over $\ground[\Ob(C)]$, and thus a dg category over $\ground$ by Proposition \ref{prop:comodule}.
\end{proof}
\begin{rem}
Given a pointed curved coalgebra $C$ and a dg category $\cat D$ we can write the MC elements from Definition \ref{def:semialgebramc} explicitly as follows:
 \begin{align*}
&\MC(C, \cat D) \cong \coprod_{O: \Ob(C) \to \Ob(\D)} 
\MC(\Hom_{\ground[\Ob(\cat D)]}(\bar C, V_{\cat D})) \\
&\cong \coprod_{O: \Ob(C) \to \Ob(\D)} 
\MC \left[\prod_{s, t\in\Ob(C)} \Hom_{\ground}\left(\ground .e_{s} \boxempty _{\ground[\Ob(C)]} \bar C \boxempty_{\ground[\Ob(C)]} \ground.e_{t}, \Hom_{\cat D}(O(s),O(t))\right) \right] \\
&\cong \coprod_{O: \Ob(C) \to \Ob(\D)} 
\MC \left[\prod_{s, t\in\Ob(C)} 
e_{s} \bar C^{*} e_{t} \otimes_{\ground} \Hom_{\cat D}(O(s),O(t)) \right]. 
\end{align*}

Here we abuse notation and write $e_{s}$ and $e_{t}$ for the canonical elements in $C$ and in $C^{*}$ corresponding to objects $s$ and $t$ in $\Ob(C)$. 

The first isomorphism follows as coalgebra morphisms $\ground[\Ob(C)] \to \ground[\Ob(\D)]$ are in bijective correspondence with maps on the sets of grouplike elements.

The second isomorphism holds as we can use $k[S] = \oplus_{s \in S} \ \ground .e_{t}$ to rewrite \[
\bar C = \oplus_{s,t \in \Ob(C)} \ \ground .e_{s} \boxempty _{\ground[R]} \bar C \boxempty_{\ground[R]} \ground.e_{t}
\]
and unravel the definition of bicomodule maps.

The last isomorphism is obtained by dualizing.
\end{rem}

\subsection{Model structure}\label{sect:coalgebramodel}

We would like to promote the adjunction in Proposition \ref{prop:catadjunction} to a Quillen equivalence between the Dwyer-Kan model structure on $\dgCat$ and a suitable model structure on $\ptdco$. 
\begin{rem}
	We use here the term Dwyer-Kan model structure for the model category whose weak equivalences are quasi-equivalences \cite[Th\'eor\`eme 1.8]{Tabuada07}. The generating (acyclic) cofibrations will be recalled below. This is the first of three model structures constructed by Tabuada and the name is by analogy with the closely related Dwyer-Kan-Bergner model structure on simplicial categories.
\end{rem}

However, there is an obvious obstacle in that curved coalgebras do not have a final object. 
A related issue is that a dg category having objects whose identity morphisms are equal to zero, are not representable by split semialgebras, and thus do not admit a bar construction.

The main step in avoiding these problems is adding a final object to curved coalgebras and defining its cobar construction to be the dg category with one object and morphism space equal to zero. 

We thus formally define the symbol $*$ to be a curved coalgebra which receives exactly one morphism from every curved coalgebra and has no outgoing morphisms (except for a unique endomorphism). 

We also declare $*$ with its identity map to be a split curved coalgebra.
(While $*$ has no maximal cosemisimple subcoalgebra, it has a unique such admitting a splitting, namely itself.) 

For any split curved coalgebra $C$ we declare that there is a unique map $C \to *$ and no map going the other way.\footnote{
	This case is not covered by Definition \ref{def:curvedmorphism} as the curved map $C \to*$ is not of the form $(f,a)$. Therefore we make this definition by hand.}

We denote by $\cusemicoa^*$ the finalized category 
of split curved  coalgebras, whose objects are all split curved coalgebras together with $*$ and whose morphisms are as above.

We consider the full subcategory $\ptdco^*$ of $\cusemicoa^*$ consisting of all pointed curved coalgebras together with  $*$.
The construction of products of pointed curved coalgebras is somewhat subtle. To explain it, we dualize and consider the equivalent problem
of constructing coproducts of pointed curved pseudocompact algebras. For two graded pseudocompact algebras $A$ and $B$ we will write $A\coprod B$ for their coproduct in the pseudocompact category (i.e. the completion of their ordinary coproduct). Now suppose that $A, B$ are pointed and split; the splitting allows us to view their maximal semisimple quotients  $A_0$ and $B_0$  as subalgebras of $A$ and $B$. We denote by $A\coprod_{A_0,B_0} B$ the quotient of $A\coprod B$ by the closed two-sided ideal generated by the elements $a_0b_0-b_0a_0$ (i.e. we impose the relation that the subalgebras $A_0$ and $B_0$ commute).

\begin{lem}\label{lem:coproducts}\begin{enumerate}
		\item 
Let $(A,d_A)$ and $(B,d_B)$ be pointed split pseudocompact dg algebras. Then their categorical coproduct is isomorphic to  $A\coprod_{A_0,B_0} B$. The differential is induced by the differentials $d_A$ and $d_B$ in $A$ and $B$.
\item Let $(A,d_A,h_A)$ and $(B,d_B,h_B)$ be pointed curved pseudocompact  algebras. Then their categorical coproduct is the pointed curved pseudocompact algebra  $A\coprod_{A_0,B_0} B\langle \langle c\rangle\rangle$ where $c$ is a generator in degree 1. 	The differential is given by $a \mapsto d_A a$ for $a\in A$, $b \mapsto d_B b-[c,b]$  for $b\in B$ and $c \mapsto h_B - h_A - c^2$ and the curvature element is $h_A$.
\end{enumerate}	
\end{lem}
\begin{proof}
	Let us prove (1). First note that $A\coprod B$ is pointed. Indeed, let $C$ be a simple finite dimensional algebra and consider a surjective map $A\coprod B\to C$. 
	Such a map comes from a pair of maps $A\to C$ and $B\to C$ and so it factors through $A_0\coprod B_0$ and thus, through $A_0\otimes B_0$  (which is the maximal semisimple quotient of $A_0\coprod B_0$). 
	Since $A_0\otimes B_0$ is a product of copies of $\ground$, so is $C$, which implies that $A\coprod B$ is pointed. The same argument shows that the maximal semisimple quotient of $A\coprod B$ is $A_0\otimes B_0$, and the splittings $A_0\to A, B_0\to B$ determines a splitting $A_0\otimes B_0\to A\coprod B$. 
	Thus, $A\coprod B$ is a pointed split pseudocompact dg algebra. 
	
	 Next, suppose we are given a pair of morphisms $f:A\to C$ and $g:B\to C$ into a pointed curved pseudocompact dg algebra $C$.  The maps $f$ and $g$ determine a map $f\coprod g:A\coprod B\to C$ and it is easy to see that it restricts to a map $A_0\coprod B_0\to C_0$. 
	 Since $C_0$ is commutative, the latter map factors through $A_0\otimes B_0$ and it follows that $f\coprod g$ factors through $A\coprod_{A_0,B_0} B$. Conversely, any morphism $A\coprod_{A_0,B_0} B\to C$ restricts to a pair of morphisms $A\to C$ and $B\to C$. This establishes a 1-1 correspondence between maps $A\coprod_{A_0,B_0} B\to C$ and such pairs.
	 
	The proof  of (2) follows readily, with the modification of the coproduct in (1) taking into account the curvature as in \cite[Lemma 9.2]{Positselski11}.  
	Explicitly, given maps $(f,a): A \to C$ and $(g,b): B \to C$ the map $(h, x): A \amalg_{A_0, B_0} B\langle\langle c\rangle\rangle \to C$ is defined by $h|_A = f$, $h|_B = g$ and $h(c) = b-s$ while $x: A_0 \otimes B_0 \to C$ takes $r \otimes s$ to $a(r)$.
		
	Note that this construction appears asymmetrical, but there is a unique isomorphism to the alternative construction where the differential on $A$ gets twisted and the curvature element is inherited from $B$. 
\end{proof}	
\begin{lem}
	$\ptdco^*$ is complete and cocomplete.
\end{lem}
\begin{proof}
	This does not differ much from the conilpotent case treated  in \cite[Sections 9.2, 9.3]{Positselski11}.
	
	We consider limits first.
	It is enough to compute equalizers and products in $\ptdco^*$, which we do by considering the dual category.
	
	So let $A$ and $B$ be pointed curved pseudocompact algebras with maximal semisimple quotients $A_0$ and $B_0$. 
	The coequalizer of $(f,a), (g,b): A \rightrightarrows B$ is the quotient of $B$ by $f(x)-g(x)$ and $a(r) - b(r)$ for $x \in A$ and $r\in A_0$. 
	This is naturally a pointed curved pseudocompact algebra over the equalizer of $a, b: A_0 \rightrightarrows B_0$.

The product of two pointed curved pseudocompact algebras is constructed in Lemma \ref{lem:coproducts}, (2) above.

	For an arbitrary non-empty product we follow the same playbook. Consider pointed curved pseudocompact algebras $A^i$ with maximal semisimple quotients $A_0^i$ and single out an index $i_0$. We define their categorical coproduct as
	\[
	 A^{i_0} \coprod_{\{A^i_0, \, i \in I\}} \left(\coprod_{i\in I \setminus\{i_0\}}A^i\langle\langle c_i\rangle\rangle\right),
	\]
	where we take the free product of all $A^i$ together with an additional generator for each $i \in I \setminus \{i_0\}$ and quotient out by the ideal generated by all commutators of elements of the $A^i_0$. 
	
	The differential is defined as follows: Let $d^i$ and $h^i$ be the differential and curvature of $A^i$.
	Then for $a \in A^{i^0}$
	we define $a \mapsto d^{{i_0}}a$ and for $a \in A^{i \neq i_0}$ we define $a  \mapsto d^{i}a - [c_i, a]$.
The differential of 	$c_i$ is $h^{i} - h^{i_0} - c_i^2$
	The curvature element is 
	$h^{i_0}$.
	
	This coproduct is a pointed curved pseudocompact algebra with maximal semisimple quotient $\bigotimes_{i\in I}A^i_0$.
	Then  given a family of curved maps $(f^i, a^i):  A^i \to C$, the map $(f,x): A^{i_0} \coprod_{\{A^i_0,\, i\in I\}} \left(\coprod_{i \neq i_0} A^i\langle\langle c_i\rangle\rangle \right) \to C$ is determined by $f|_{A^i} = f^i$ and $f(c_i) = a^i - a^{i_0}$.
	while $x$ is given by $a^{i_0}$.
	
	The empty product, i.e.\ the final object, in $\ptdco^*$ is by definition $*$.
	
	Arbitrary coproducts in $\ptdco^*$ exist and are preserved by the forgetful functor to pairs of graded vector spaces, i.e.\ we may just equip the direct sums of the underlying graded vector spaces with differential, curvature and coproduct.
	In particular the empty coproduct is given by the split curved coalgebra $(0,0)$.
	
	To consider coequalizers we dualize again. 
	The equalizer of two pointed curved pseudocompact algebras can be computed as follows. 
	Let $(f,x), (g,y): (A, u) \to (B,v)$ be an equalizer diagram.
	We first define $k: (E, r) \to (A,u)$ to be the equalizer of the underlying diagram of graded split pseudocompact algebras. In particular $E$ is a subalgebra of $A$ and $fk = gk: E \to B$.
	For this to have a curved structure we need to find $z \in A$ such that $x + fz = y + gz$.
	If no such $z$ exists then there can't be any map from a curved algebra equalizing $(f,x)$ and $(g,y)$. So there is no equalizer in curved pointed pseudocompact algebras and in the initialized category the equalizer is the initial object.
	
	 If such a $z$ exists then we define $d_E(e) = d_A(e) + [z, e]$ and $h_E = h_A + dz + z^2$ and $(k, z): (E, r) \to (A,u)$ equalizes $(f,x)$ and $(g,y)$. (Here we denote differentials and curvature terms in the different algebras by subscript.)
	The cone $(k,z)$ is universal as for any other $(k', z'): (E', r') \to (A,u)$ we know $k'$ factors through $k$ by considering the underlying gradeds. So we may assume $k= k'$, $E =E'$ and just consider the case $z \neq z'$. Then $z-z'$ is in $E$ and $(\id_E, z-z')$ is a unique isomorphism from $(k,z')$ to $(k,z)$, showing we have indeed constructed the limit.
	
	Thus $\ptdco^*$ is complete and cocomplete.
\end{proof}

We define the cobar construction of $*$ to be the dg category with one object and only the zero morphism, which we will call the zero dg category.

We similarly define the bar construction of the zero dg category as $*$.

There are other dg categories which have some identity morphisms equal to zero.
Instead of defining new curved coalgebras to be their bar constructions, we will define a new model category $\dgCat'$.
It is the full subcategory of $\dgCat$ which consists of all dg categories where no identity morphism equals 0, together with the zero dg category. 

To put a model structure on $\dgCat'$ we first recall from \cite[Th\'eor\`eme 1.8]{Tabuada07} the sets $\mathcal I$ and $\mathcal J$ of generating cofibrations and generating acyclic cofibrations in $\dgCat$:

 By abuse of notation let $\emptyset$ be the empty dg category and let $\mathcal A$ be the dg category with a single object $*$ and $\Hom(*,*) = \ground$. Let $e: \emptyset \to \mathcal A$ be the unique functor.
For each $n \in \mathbb Z$ let $D(n)$ be the dg category with two objects 1 and 2 and with $\Hom(1,1)=\Hom(2,2)= \ground$, $\Hom(1,2) = \operatorname{cone}(\id: \ground[-n] \to \ground[-n])$ and $\Hom(2,1)=0$. 
For each $n \in \mathbb Z$ let $S(n)$ be the dg category which only differs from $D(n)$ by $\Hom(1,2) =  \ground[-n]$ and let $i_n: S(n) \to D(n)$ be the natural inclusion functor.

Then $\mathcal I$ is the set consisting of $e$ and all $i_n$.

Let $\mathcal B$ be the dg category with two objects $1$ and $2$ that only differs from $D(n)$ in that $\Hom(1,2) = 0$. Let $\mathcal K$ be the dg category with two objects $1$, $2$ and morphisms generated by the identities as well as $f \in \Hom^0(1,2)$, $g \in \Hom^0(2,1)$, $r_1 \in \Hom^{-1}(1,1)$, $r_2 \in \Hom^{-1}(2,2)$ and $r_{12} \in \Hom^{-2}(1,2)$ satisfying $df = dg =0$, $dr_1 = gf - \id_1$, $dr-2 = fg - \id_2$ and $dr_{12} = fr_1 - r_2f$.
Let $k$ be the natural inclusion $\mathcal A \to \mathcal K$ sending 1 to 1 and let $j_n$ be the natural inclusion $\mathcal B \to D(n)$.

Then $\mathcal J$ is the set consisting of $k$ and all $j_n$.

\begin{lem}
	The generating cofibrations $\mathcal I$ of $\dgCat$ turn into $\dgCat'$ a model category.
\end{lem}
\begin{proof}
	We note first that dg categories with identities equal to zero do not feature in the generating cofibrations and generating trivial cofibrations of $\dgCat$, thus $\mathcal I$ and $\mathcal J$ are contained in $\dgCat'$.
	Thus we define a model structure on $\dgCat'$ whose morphisms are cofibrations, weak equivalences and fibrations exactly if they are cofibrations, weak equivalences and fibrations in $\dgCat$.
	
	We will check that $\dgCat'$ is closed under limits and colimits. 
	Then we can run the small-object argument restricted to the subcategory $\dgCat'$ to prove the existence of factorizations. All other model category axioms are inherited by full subcategories.
	
	We note that the product $\prod_i \cat D_i$ in $\dgCat'$ exists. 
	It is equal to the product of all factors that are not the final object.
	
	For two functor $F, G: \cat C \rightrightarrows \cat D$ in $\dgCat'$ we denote the equalizer in $\dgCat$ by $J: \cat B \to \cat C$. 
	All identities in $\cat B$ are nonzero (as $FJ(\id_B)) = GJ(\id_B)$ for $B \in \cat B$), unless they are zero in $\cat C$, in which case $\cat C$, $\cat D$ and $\cat B$ are zero.
	
	The coproduct $\amalg_i \cat D_i$ of objects in $\dgCat'$ is the same as in $\dgCat$. It is zero if one of the $\cat D_i$ is, and has no zero identities otherwise.
	
	For the coequalizer a similar argument applies.
\end{proof}

It is easy to see that $\Omega: \ptdco^* \rightleftarrows \dgCat': \Ba$ is an adjunction and we will define a model structure on $\ptdco^*$ so that this becomes a Quillen equivalence. 

\begin{lem}\label{lem-spcoapresentable}
	The category $\ptdco^*$ is locally presentable.
\end{lem}
\begin{proof}
	We recall that an object $s$ of a category is finitely presentable if $\Hom(s,-)$ commutes with filtered colimits.
	A category $\cat C$ is finitely presentable (and thus in particular locally presentable) if it is cocomplete and there is a small set of finitely presentable objects $S$ such that every object of $\cat C$ is a filtered colimit of objects in $S$. 
	In particular $S$ is a generating set.
	
	We let $S \subset \ptdco$ be the subcategory of finite-dimensional split coalgebras.
	It is clear that any finite-dimensional split curved coalgebra is finitely presentable.
	
	Next, any coalgebra is the union of its finite-dimensional sub-coalgebras.
	The same applies to dg coalgebras and for curved coalgebras: 
	it is clear for graded coalgebras and the coalgebra obtained by adding the images of all differentials to a finite-dimensional subcoalgebra is still finite-dimensional (using that the action of the curvature term is nilpotent).
	
	We now claim that any pointed curved coalgebra $(C,\epsilon)$ is a colimit of finite-dimensional ones. 
	Indeed, $C$ is the union of its finite-dimensional subcoalgebras $C^i$, and the restriction of $\epsilon$ provides splittings $\epsilon^i: C^i \to C_0\cap C^i = (C^i)_0$ compatible with inclusions.
	Then $(C, \epsilon)$ is the colimit of the $(C^i, \epsilon^i)$.
	Thus $S$ generates $\ptdco$ under filtered colimits.
	
		Thus $\ptdco$ is locally presentable, and the lemma follows as the additional object $*$ is finitely presentable.
\end{proof}

\begin{prop}\label{prop:coalgebramodel}
	The category $\ptdco^*$ admits a left proper combinatorial model structure whose weak equivalences are maps $f$ such that $\Omega(f)$ is a quasi-equivalence and whose  cofibrations are all inclusions together with the map $0 \to *$.
\end{prop}
\begin{proof}
	This is an application of \cite[Proposition A.2.6.15]{Lurie11a}. 
	We let $W$ be the class of maps that $\Omega$ sends to quasi-equivalences,\footnote{We consider $\Omega$ as a functor into $\dgCat$ rather than $\dgCat'$ here as this is more convenient for certain arguments.}
	 and we denote by $C_0$ a set representing all inclusions of finite-dimensional pointed curved coalgebras. 
	The existence of the model structure then follows if we have the following:
	\begin{enumerate}
		\item $\ptdco^*$ is locally presentable.
		\item The class $W$  is stable under filtered colimits.
		\item There is a small set $W_0$ contained in $W$ such that all morphisms in $W$ are filtered colimits of elements in $W_0$.
		\item The class $W$ is stable under pushouts along pushouts of morphisms in $C_0$.
		\item Any morphism with the right lifting property with respect to all morphisms in $C_0$ lies in $W$.
	\end{enumerate}
In the first condition note that Lurie uses \emph{presentable} for what we call locally presentable.

The second and third condition show that $W$ is perfect as it is clear that $W$ contains isomorphisms and satisfies the two-out-of-three property. Thus (1-5) will show $\ptdco$ is a model category.

The category $\ptdco^*$ is finitely presentable by Lemma \ref{lem-spcoapresentable}.

Conditions (2) and (3) are exactly saying that $W \subset {\ptdco}^{[1]}$ is a finitely accessible subcategory, cf.\  \cite[Definition A.2.6.2]{Lurie11a} for the case $\kappa = \aleph_0$.
Here $[1]$ denotes the category with 2 objects and 1 non-identity morphism.

As $\Omega$ preserves filtered colimits so does ${\ptdco^*}^{[1]} \to \dgCat^{[1]}$.
By \cite[Corollary A.2.6.5]{Lurie11a}
it then suffices to show that quasi-equivalences of dg categories are finitely accessible. 
Since $\dgCat$ is combinatorial by \cite{Holstein3} this follows from \cite[Corollary A.2.6.6]{Lurie11a}. 
We need to strengthen the conclusion a little bit, from $W$ being accessible to being finitely accessible. 
Inspecting the proof (as well as the proof of \cite[Proposition A.1.2.5]{Lurie11a}) this follows if the sources and targets of all generating cofibrations and generating acyclic cofibrations are $\aleph_0$-compact, i.e. morphisms out of them commute with filtered colimits. 
This is immediate for the sets $\mathcal I$ and $\mathcal J$ of generating (acyclic) cofibrations for $\dgCat$.

For (4) we first show that $\Omega$ sends inclusions of pointed curved coalgebras to cofibrations. 
Let $C \to D$ be an inclusion in $\ptdco$. We need to show that $\Omega(D)$ is of the form $\colim_{i \geq 0} G^i$ where $G^0 = \Omega(C)$ and $G^{i+1}$ is obtained from $G^i$ by freely adjoining morphisms $f$ with differentials $df \in G^i$.

It is clear that the cobar construction on $D$ corresponds to the free category generated by the elements of $D$. 

Thus we use the filtration induced by the order of comultiplication on $D$, i.e.\ $F^0 = \ground[\Ob(D)]$ and $F^i = \Delta^{-1}(F^{i-1}(D \otimes D))$.
Then $\Omega(C + F^{i+1})$ is obtained from $\Omega(C + F^i)$ by first freely adjoining arrows for all the homogeneous elements of $F^{i+1}$ with differential equal to 0, and then for all other homogeneous elements of $F^{i+1}$.
This shows $\Omega(C) \to \Omega(D)$ is a cofibration in $\dgCat$.

Now let $f: X \to Y$ be a morphism in $C_0$ and $f': X' \to Y'$ its pushout along an arbitrary map $X \to X'$.
We have shown that $\Omega$ sends maps in $C_0$ to cofibrations.
As $\Omega$ is a left adjoint it preserves pushouts and $\Omega(f')$ is a pushout of $\Omega(f)$ and thus also a cofibration.

If now $g: X' \to X'''$ is in $W$ and $g'': Y' \to Y''$ is its pushout along $f'$ we need to to show that $g''$ is also in $W$.
By assumption $\Omega(g)$ is a quasi-equivalence and $\Omega(f')$ is a cofibration. 
Since $\dgCat$ is left proper if $\ground$ is a field, see \cite{Holstein3}, we deduce that $\Omega(g')$ is a quasi-equivalence and $g'$ is in $W$.

Before showing (5) we will prove, using the representation of split curved coalgebras as unions of their finite-dimensional sub-coalgebras that all inclusions in $\ptdco^*$ are obtained by pushout and transfinite composition from the morhphisms in $C_0$
\footnote{We are grateful to L. Positselski for providing this argument to us.} Given an inclusion of split curved coalgebras $X\to Y$, choose an element $y\in Y$ such that $y\notin X$ and a finite dimensional split curved coalgebra $Z$ containing $y$. 
It follows that the inclusion $X \to X+Z$ 
has finite-dimensional cokernel; call such inclusions \emph{good}. It follows by transfinite induction that any inclusion of split curved coalgebras is a transfinite composition of good inclusions. Let now $U\to V$ be a good inclusion. Choosing a finite basis in $V/U$, lifting it arbitrarily to $V$ and choosing a finite dimensional split curved coalgebra containing the lift of every basis element, we find a finite dimensional subcoalgebra $W$ in $V$ such that $U+W=V$. Then the map $U\to V$ is obtained as a pushout of $U\bigcap W\to U$ along the inclusion of finite dimensional split curved coalgebra $U\bigcap W\to W$.  
In particular this will show that the cofibrations in $\ptdco$, generated by $C_0$, are exactly inclusions.

Finally, for (5) it follows by a standard argument going  back to Quillen that if all objects are mapped to cofibrant objects by $\Omega$ and there exist cylinder objects in $\ptdco$, it follows that any map with the RLP with respect to preimages of cofibrations is a weak equivalence. 
The result in the form we use can be found in \cite[Theorem 2.2.1]{Hess17}.

We prove for any split curved coalgebra $C$ the existence of a cylinder object, i.e.\ some factorization $C \amalg C \xrightarrow{i} C \otimes I \xrightarrow{q} C$ of the fold map such that $\Omega(i)$ is a cofibration and $\Omega(q)$ is a quasi-equivalence. We let $\mathcal C$ be the dg category with two objects connected by an isomorphism. It has a natural quasi-equivalent projection to the dg category $\mathcal A$ with one object whose algebra of endomorphisms is $\ground$. There is a natural tensor product in $\ptdco$ that makes $C \otimes \Ba(\mathcal C)$ into a coalgebra. This is a cylinder object for $C$. Indeed, there is a natural inclusion from $C \oplus C$ and the map $C \otimes \Ba(\mathcal C) \to C$ induced by $\Ba(\mathcal C) \to \Ba(\mathcal A) = \ground$ is a weak equivalence. To justify the last claim, we can use the general result that $\Omega$ is quasi-strong monoidal, i.e. $\Omega(C\otimes B(\mathcal C)) \simeq \Omega C \otimes \Omega B (\mathcal C) \simeq \Omega C \otimes \mathcal C \simeq \Omega C$. This is \cite[Lemma 5.1]{HolsteinJ}.
Note that the result we quote is a pure computation which is not dependent on the existence of a model structure on $\ptdco^*$.

One may also show directly, that there is a cylinder object by considering the coalgebra $I$ of chains on the classifying space of $\mathcal C$ and showing there is a homotopy equivalence between $\Omega(C \otimes I)$ and $\Omega(C)$ using the symmetric monoidal structure on $\mathcal C$,
cf.\ the proof of claim 6 in \cite[Proposition 3.3]{booth23}.

This finishes the proof. We note that we could have proved the proposition equally by using the main theorem of \cite{Garner20}, the necessary checks are very similar.
\end{proof}
\begin{rem}\label{rem:hquillen}
	One may define a model category structure on pointed dg coalgebras with the same weak equivalences and cofibrations.
	
	As the inclusion $\iota$ of pointed dg coalgebras into pointed curved coalgebras then preserves cofibrations, the uncurving adjunction $\iota \dashv \Eta$ from Corollary \ref{cor:uncurvingcoalg} is Quillen.
	Note that $\Eta (*)$ is the terminal dg coalgebra $\ground$.
\end{rem}
\begin{rem}
	The weak equivalences between conilpotent curved coalgebras considered in \cite{Positselski11} are generated by \emph{filtered quasi-isomorphisms} with respect to admissible filtrations. 
	A naive generalization to pointed curved coalgebras would lead us to consider admissible filtrations with $F^0(C) = C_0$ (instead of $F^0 = \ground$ as in \cite{Positselski11}).
	Then any filtered quasi-isomorphism would have to preserve the the set of grouplike elements of a pointed curved coalgebra.
	But we know that a quasi-equivalence of dg categories can change the set of objects, thus such maps do not generate all weak equivalences of pointed curved coalgebras.
\end{rem}
\begin{rem}
	We may apply the same reasoning to  the adjunction between conilpotent curved coalgebras (together with a final object) and dg algebras to obtain a model structure on curved conilpotent coalgebras that is Quillen equivalent to the standard model structure on dg algebras.
	
	This differs from the model structure considered in \cite{Positselski11} in that for us a map $C \to *$ is only a cofibration if $C = *$ or $C = 0$.
	Moreover, there are curved coalgebras weakly equivalent to $*$, namely those arising as bar constructions of dg algebras with 0 homology.
\end{rem}
\begin{rem}
	One should note that from a homotopy-theoretic point of view the final objects $*$ and zero are not very important, and neither is the difference between $\dgCat$ and $\dgCat'$.
	
	Recall that we add a final object to curved coalgebras to satisfy the requirements of a model category to have all limits.
	However, this final object is weakly equivalent to an honest split curved coalgebra, just like the zero dg category is quasi-equivalent to a dg category with one object whose identity is a coboundary.
	Thus the $\infty$-categories obtained by localizing at all weak equivalences in $\ptdco^*$ and $\ptdco$ are in fact weakly equivalent.
	In particular the non-finalised category $\ptdco$ has all homotopy limits (or equivalently all $\infty$-categorical limits).
	
	Similarly, $\dgCat$ and $\dgCat'$ give rise to the same $\infty$-category of dg categories.
	
	Thus it is left to the reader's taste if they want to consider the strictly (rather than homotopy) terminal curved coalgebra, and if they want to consider the zero dg category and other dg categories with identities equal to zero.
\end{rem}
\begin{rem}
	Recall that in the model category of conilpotent curved coalgebras the fibrant objects are precisely those whose underlying coalgebras are conilpotent cofree (i.e. whose dual pseudocompact algebras are completed tensor algebras). This follows from a more general result of \cite[Section 9.3, Lemma 2(i)]{Positselski11} describing fibrations in this model category. These fibrant objects are nothing but (unital) $A_\infty$ algebras. Arguing similarly to op.\ cit.\ we can obtain an explicit description of fibrant objects in $\ptdco^*$; these are precisely those whose underlying pointed coalgebras are of the form $T_R(V)$; the cotensor algebra over a pointed cosemisimple coalgebra $R$ of an $R$-bimodule $V$. These objects are otherwise known as (unital) $A_\infty$ categories; these come up in various contexts of homological algebra and geometry, particularly in the study of Fukaya categories \cite{Sei08}.   
\end{rem}

\subsection{Quillen equivalence}\label{sect:quillenequivalence}
\begin{lem}\label{lem-counit}
For any small dg category we have a quasi-equivalence $\Omega \Ba \cat D \simeq \cat D$. 
\end{lem}
\begin{proof}
	If all hom spaces in $\cat D$ are 0 this follows directly from the definitions. 
	So let us assume $B$ is given by the reduced bar construction.
	
By construction $\Omega \Ba \cat D$ and $\cat D$ have the same set of objects. 
Thus the counit $\Omega \Ba \cat D \to \cat D$ is a quasi-equivalence if for any $A, B \in \Ob(\D)$ we have a quasi-isomorphism $\Hom_{\Omega \Ba \D}(A, B) \simeq \Hom_{\D}(A, B)$.
It suffices to show that the semialgebras $\Omega \Ba \D$ and $\D$ are quasi-isomorphic, as we can recover all morphism spaces as in Proposition \ref{prop:comodule} via $\Hom(A,B)=\ground_{A}\boxempty_{\ground[\Ob(\D)]}V_{\cat D}\boxempty_{\ground[\Ob(D))]}{\ground_{B}}$

We apply the usual argument that the cobar-bar counit is a quasi-isomorphism, just  in the category of $\ground[\Ob(\cat D)]$-bicomodules,  see e.g. \cite[Theorem 6.10]{Positselski11}.

Consider the two-stage filtration on $\cat D$: $$\ground[\Ob(\cat D)] \subset \cat D.$$ The induced filtration on $\Omega \Ba(\D)$ is multiplicative and compatible with the differential. Note that the differential in $\operatorname{gr}\Omega \Ba(\D)$ has no contributions from the bar differential and curvature in $\Ba(\D)$, in other words the semialgebra $\operatorname{gr}\Omega \Ba(\D)$ is just the cobar construction of a pointed bigraded dg coalgebra whose underlying space and coproduct are the same as  $\Ba(\D)$ but the curvature and the bar differential are set to zero.

The counit of the cobar-bar adjunction $\Omega \Ba(\cat D)\to D$ is compatible with these filtrations, and so, it gives a map on the associated spectral sequences. The spectral sequence associated with the two-stage filtration on $\cat D$ is trivial (has zero differentials) whereas the differential $d_0$ of the filtration on $\Omega \Ba(\cat D)$ is just the relative cobar-differential of the coalgebra  $\Ba(\cat D)$ viewed as a coalgebra relative to $\ground[\Ob(\cat D)]$ and disregarding the bar differential and curvature of $\Ba(\cat D)$.
 The latter cobar-construction computes the functor $\operatorname{Ext}^*(\ground[\Ob(\cat D)], \ground[\Ob(\cat D))$ in the category of graded comodules over $\Ba(\cat D)$ (where it is understood that the bar differential and curvature in $\Ba(\cat D)$ is disregarded). Choosing appropriate bases in the morphism spaces of $\cat D$, we identify $\Ba(\cat D)$ with the path coalgebra of a certain quiver (cf. Example \ref{quiver}). Since path coalgebras are hereditary, the cohomology of this cobar-construction is zero in degrees $>1$.  In degrees 0 and 1, on the other hand, it equals $\ground[\Ob(\cat D)]$ and $\bar{D}$ respectively.

It follows that the corresponding spectral sequences are isomorphic from the term $E_1$ onwards which gives the desired statement.
\end{proof}
\begin{theorem}\label{thm:koszulquillen}
The adjunction $\Omega \dashv \Ba$ induces a Quillen equivalence between $\ptdco^*$ and $\dgCat^\prime$ with the model structures from Section \ref{sect:coalgebramodel}.
\end{theorem}
\begin{proof}
$\Omega$ is left Quillen as it sends generating cofibrations to cofibrations, see the proof of Proposition \ref{prop:coalgebramodel}.

We have shown in Lemma \ref{lem-counit} that the counit is a weak equivalence. Let now $C$ be any pointed curved coalgebra.
Since $\Omega C \to \Omega \Ba \Omega C \to \Omega C$ is the identity by the triangle identity of the adjunction it also follows by 2-out-of-3 that the natural map $\Omega C \to \Omega \Ba \Omega C$ is a quasi-equivalence, and thus the unit $C \to \Ba \Omega C$ is a weak equivalence.
\end{proof}

If we are relaxing our assumption that the ground ring $\ground$ is a field, the naive bar construction is no longer well-behaved.
We still have a version of our main result if we let $\ground$ be a principal ideal domain (the case of interest is of course $\ground = \mathbb Z$) and restrict to the category $\dgCat_{\operatorname{fr}}$ of dg categories such that the underlying graded modules for all hom spaces are free.

The second is not a serious restriction as any dg category may be canonically replaced by a quasi-equivalent \emph{semi-free} dg category, see Lemma B.5 in \cite{Drinfeld04} and the preceding definition.

Similarly we may consider Definition \ref{def:semicoalgebraworking} over $\ground$ and consider
the category $\ptdco^*_{\operatorname{fr}}$ of split curved coalgebras whose underlying graded $\ground$-module is free. 

In this setting we define the functors $B$ and $\Omega$ on $\dgCat_{\operatorname{fr}}$ and $\ptdco^*_{\operatorname{fr}}$
as in Definitions \ref{def:semibar} and \ref{def:semicobar}.

We will consider $\dgCat_{\operatorname{fr}}$ and $\ptdco^*_{\operatorname{fr}}$ as relative categories. To do this we  declare the quasi-equivalences in $\dgCat_{\operatorname{fr}}$ to be weak equivalences, and define a morphism $f: C \to D$ in $\ptdco_{\operatorname{fr}}$ to be a weak equivalence if $\Omega(f)$ is.

\begin{cor}\label{cor:koszulz}
	Let $\ground$ be a principal ideal domain.
	Then with notation as above there is a weak equivalence of relative categories between $(\ptdco^*_{\operatorname{fr}}, \simeq)$ and $(\dgCat_{\operatorname{fr}}, \simeq)$.
\end{cor}
\begin{proof}
	This follows as in \cite[Proposition 3.7]{HolsteinH}:
	Let $\cat D$ be an object of $\dgCat_{\operatorname{fr}}$.
	Then tensoring over $\ground$ with any field commutes past the bar and cobar constructions.
	Thus the natural morphism $\eta_{\cat D}:\Omega \Ba \cat D \to \cat D$ becomes a quasi-isomorphism (of bicomodules) after tensoring with any field, thus $\eta_{\cat D}$ must have been a quasi-isomorphism, and thus a quasi-equivalence of dg categories, by \cite[Lemma 3.6]{HolsteinH}.
	
	As in the proof of Theorem \ref{thm:koszulquillen} it follows that $C \simeq \Ba \Omega C$ and together this gives a strict homotopy equivalence of relative categories, and thus a weak equivalence.
\end{proof}

\subsection{Semimodule-comodule level Koszul duality}\label{sect:comodulekoszul}
 The bar-cobar adjunction \ref{thm-semibarcobar} gives rise to an equivalence between the corresponding derived and coderived categories as in the case of ordinary dg Koszul duality, cf. \cite[Theorem 6.3, 6.4]{Positselski11}. We now formulate a generalization of this result with dg algebras replaced by dg categories and conilpotent curved coalgebras replaced by pointed curved coalgebras. The treatment of \cite[Chapter 6]{Positselski11}, carries through with fairly obvious modifications; namely the tensor product over $\ground$ needs to be consistently replaced with the cotensor product over a suitable  cosemisimple coalgebra. 

Here we only consider semialgebras of the form $(A,R)$ where $R$ is a cosemisimple coalgebra. Recall, first of all, that the category of semimodules over a such a semialgebra is a model category where weak equivalences are quasi-isomorphisms and fibrations are surjective maps \cite[Theorem 9.2]{Posit10}. Note that the result in op.\ cit.\ was formulated in greater generality, particularly not requiring that $R$ be semisimple, and with this simplification, Positselski's notion of a weak equivalence reduces to a quasi-isomorphism. Similarly the category of comodules over a (not necessarily conilpotent) coalgebra is a model category where weak equivalences are maps with a coacyclic cone and fibrations are injective maps, \cite[Theorem 8.2]{Positselski11}.

Let $C$ be a split curved coalgebra  and $(A,C_0)$ a split semialgebra. 
Assume that the curved convolution algebra $\Hom_{C_0}(\bar{C}, A)$ possesses an MC element $\tau$.

Recall the categories of $C$-comodules from Definition \ref{def:comodule} and $(A,C_0)$-semimodules from Definition \ref{def:semimodule}.
We construct a functor associating to a $C$-comodule $N$ a dg $A$-semimodule $A\boxempty_{R}^\tau N$ as follows. 
The underlying $A$-semimodule of $A\boxempty_{R}^\tau N$  is $A\boxempty_{R}N$ with $A$ acting freely on the left, whereas the  differential $d^{\tau}$ is given by the formula
\[
d^\tau(x\boxempty n)=d(x\boxempty n)-x\boxempty(\tau\otimes 1)\Delta(n)
\]
where $x\in A, n\in M, \Delta:N\to C\otimes N$ is the coaction on $N$ and $d$ stands for the ordinary differential on $N$ induced by the differential on  $N$. Thus, $A\boxempty_{R}^\tau N$  is given by cotensoring $M$ with $A$ and twisting the differential by the MC element $\tau$.

Similarly we construct a functor associating to an $(A,C_0)$-semimodule $M$ a $C$-comodule $C\boxempty_{C_0}^\tau M$ as follows. Disregarding the differential, it is $C\boxempty_{C_0} M$ with $C$ coacting cofreely on the left, whereas the differential $d^{\tau}$ is given by the formula

\[
d^\tau(c\boxempty m)=d(c\boxempty m)+(1\otimes\tau)\Delta(c)\boxempty m
\]
where $c\in C, m\in M, \Delta:C\to C\otimes C$ is the diagonal on $C$ and $d$ stands for the ordinary differential on $C\boxempty_{C_0}M$ induced by the differentials on $C$ and $M$. Thus, $C\boxempty_{C_0}^\tau M$ is given by cotensoring $M$ with $C$ and twisting the differential by $\tau$.

Clearly, $N\mapsto A\boxempty_{C_0}^\tau N$ is a dg functor $C\Comod\to (A,C_0)\Mod$ and $M\mapsto C\boxempty_{C_0}^\tau M$ is a dg functor $(A,C_0)\Mod\to C\Comod$.
\begin{prop}\label{prop:fgadjoint}
The functor $A\boxempty_{C_0}^\tau -$ is left adjoint to $C\boxempty_{C_0}^\tau -$.	
\end{prop}
\begin{proof}	One simply has to note that the dg spaces of morphisms $\Hom_{A\Mod}(A\boxempty_{C_0}^\tau N,M)$ and $\Hom_{C\Comod}(N,C\boxempty_{C_0}^\tau M)$ are naturally isomorphic to $\Hom_{C_0}(N,M)$ with the differential twisted by the MC element $\tau$.
\end{proof}	
We apply these constructions to two situations:
\begin{enumerate}
	\item Given a split semialgebra $(A,C_0)$, the split curved coalgebra $C$ is the bar-construction of $A$, $C=\Ba(A)$, so that 
	the MC element $\tau\in\Hom_{C_0}(\bar{C},A)$ is the one corresponding to the identity map $\Ba(A)\to \Ba(A)$ via the adjunction in Theorem \ref{thm-semibarcobar}.	
	
\item 	Given a split curved coalgebra $C$, the semialgebra $(A,C_0)$ is its cobar-construction, $A=\Omega C$ and the MC element $\tau\in\Hom_{C_0}(\bar{C},\Omega C)$ is the one corresponding to the identity map $\Omega C\to \Omega C$ via the adjunction \ref{thm-semibarcobar}.
\end{enumerate}	
Then the following result holds.
\begin{theorem}\label{thm:modulecomodule}\
	\begin{enumerate}
		\item 
	Let $(A,C_0)$ be a split semialgebra. Then the functors 
	\[
	A\boxempty_{C_0}^\tau -:\Ba(A)\Comod\rightleftarrows A\Mod:\Ba A\boxempty^\tau_{C_0}-
	\]
	form  a Quillen equivalence.
	\item Let $C$ be a split curved  coalgebra.
	Then the functors
	\[
	\Omega C\boxempty_{C_0}^\tau-:C\Comod\rightleftarrows \Omega(C)\Mod:C\boxempty_{C_0}^\tau-
	\]
	form a Quillen equivalence.
	\end{enumerate}
\end{theorem}
\begin{proof}
	
	The arguments of \cite[Theorems 6.3, 6.4]{Positselski11} carry over to our situation without any changes (other than replacing the ground field by $C_0$)to show that in  both cases (1) and (2) both adjoint functorsu
	\begin{itemize}\item
	 preserve weak equivalences (and so descend to the homotopy categories) and
	\item determine mutual equivalences of homotopy categories.
	\end{itemize}
Note that, also in both cases, the right adjoint functor  preserves monomorphisms and so, in particular, takes cofibrations to cofibrations. Together these facts imply that, in both cases, the given adjunction forms a Quillen equivalence of model categories.
\end{proof}
\begin{example}\label{eg:representables}
Let $\cat D$ be a dg category and $B\cat D$ be its bar construction.
For any object $X$ in $\cat D$ we consider the 1-dimensional $B\cat D$-comodule $\ground_X$ generated by the grouplike element in $B\cat D$ corresponding to $X$.
Unravelling definitions, we see that $\cat D \boxempty^\tau_{\ground[\Ob(\cat D)]} \ground_X$ is the functor (co)represented by $X$.
Thus the image under the dual Yoneda embedding of $\cat D^{\op}$ in $\cat D\Mod$ is identified with the dg category of 1-dimensional comodules in $B\cat D\Comod$ (equivalently the Yoneda image of $\cat D$ is identified with dg category of 1-dimensional \emph{right} comodules in $B\cat D$-comodules). 
\end{example}	
\begin{cor} 
	If two split curved coalgebras are weakly equivalent, then their coderived categories are equivalent.
\end{cor}
\begin{proof}
Let $f:C\to D$ be a weak equivalence between split curved coalgebras $C$ and $D$. Recall from \cite[Section 4.8]{Positselski11} that $f$ induces a pair of adjoint functors $E_f:C\Comod\leftrightarrows D\Comod:R_f$. Here for a $D$-comodule $M$ we have $E_f(M)=C\boxempty_DM$ and $R_f$ is the restriction of scalars from $C$ to $D$. Consider the following diagram of dg categories and functors.
\begin{equation}\label{eq:restrictioncorestriction}
\xymatrix@=42pt{
	C\Comod\ar@<-0.5ex>[d]_{\Omega C\boxempty^\tau_{C_0}-}
	\ar@<-0.5ex>[r]_{E_f}
	&D\Comod\ar@<-0.5ex>[l]_{R_f}
	\ar@<-0.5ex>[d]_{\Omega D\boxempty^\tau_{D_0}-}\\
	(\Omega(C), C_0)\Mod
\ar@<-0.5ex>_{C\boxempty_{C_0}^\tau-}[u]
\ar@<-0.5ex>_{\Omega(E_f)}[r]
&(\Omega(D),D_0)\Mod
\ar@<-0.5ex>_{\Omega(R_f)}[l]
\ar@<-0.5ex>_{D\boxempty_{D_0}^\tau-}[u]
}
\end{equation}
Then straightforward inspection shows that it commutative in the sense that there exist natural isomorphisms $E_f\circ \Omega D\boxempty^\tau_{D_0}-\cong \Omega(E_f)\circ\Omega C\boxempty^\tau_{C_0}-$ and $C\boxempty_{C_0}^\tau-\circ\Omega(R_f)\cong R_f\circ D\boxempty_{D_0}^\tau-$.

Since $f$ is a weak equivalence, the functors $\Omega(E_f)$ and $\Omega(R_f)$ induce an adjoint equivalence between the derived categories
$D(\Omega(C))$ and $D(\Omega(D))$. It follows from (\ref{eq:restrictioncorestriction}) and Theorem \ref{thm:modulecomodule} that $E_f$ and $R_f$ induce an adjoint equivalences between $D^{\co}(C)$ and $D^{\co}(D)$ as claimed.
\end{proof}	
We will also need Theorem \ref{thm:modulecomodule} as a statement about $\infty$-categories. 
\begin{defi}\label{defi:derivedinfinitysecond}
	The \emph{coderived $\infty$-category} of a coalgebra $C$, written $\Dco(C)$ is the quasicategory obtained by localizing the category of $C$-comodules at all maps with coacyclic cone.
	
	Similarly we write $\mathscr D(A)$ for the derived $\infty$-category, obtained by localizing the category of $A$-semimodules at all quasi-isomorphisms.
\end{defi}
In particular if $\mathcal D$ is a dg category then $\mathscr D(\cat D)$ stands for the $\infty$-category of functors into $\Ch$, localized at object-wise quasi-isomorphisms.

Then we have $\Dco(C) \simeq \mathscr D(\Omega(C))$ and  $\Dco(\Ba(A)) \simeq \mathscr D(A)$ in the setting of Theorem \ref{thm:modulecomodule}.

\begin{rem} 
It may be interesting to consider an analogue of \emph{nonconilpotent} Koszul duality, cf. \cite[Section 6.7]{Positselski11} where the starting point is a curved relative coalgebra  $(C,R)$ where $R$ is a not necessarily the coradical of $C$. In this situation there is still an adjoint pair $\Omega(C)\Mod\rightleftarrows C\Comod$ but it is not a Quillen equivalence, in general (with the standard model structure on $\Omega C\Mod$). 
Indeed, already for $R=\ground$ one has to consider exotic weak equivalences (of second kind) on the side of dg-modules. We will not treat this case.
\end{rem}

\section{The dg nerve and its adjoint}\label{sect:dgnerve}
In this section we will revisit the construction of the dg nerve $\Ndg(\C)$ of a dg category $\C$ from \cite{Lurie11} and explicitly describe its left adjoint, using our bar-cobar adjunction between dg categories and pointed curved coalgebras. 

To work in maximal generality we allow $\ground = \mathbb Z$ in this section, but we will mention some results specific to the case when $\ground$ is a field.

We first recall Definition 1.3.1.6 from \cite{Lurie11}. 
\begin{defi}\label{def:lurienerve}
	Given a dg category $\D$ we define its \emph{differential graded nerve} $\Ndg(\D)$ as a simplicial set as follows.
	For all $n \geq 0$ we let $\Ndg(\D)_{n} = \Hom_{\sSet}(\Delta^{n}, \Ndg(\D))$ be the set of ordered pair $(\{X_{i}\}_{0 \leq i \leq n}, \{f_{I}\})$ where:
	\begin{enumerate}
		\item For $0 \leq i \leq n$, $X_{i}$ is an object of $\D$.
		\item For every subset $I \subset \{i_{-} < i_{m} < \dots < i_{1} < i_{+}\} \subset [n]$ with $m \geq 0$, $f_{I} \in \Hom^{m}_{\D}(X_{i_{-}}, X_{i_{+}})$ satisfying the equation
		\[
		df_{I} = 
		\sum_{1 \leq j \leq m} (-1)^{j}\left( f_{I \setminus \{i_{j}\}} - 
		f_{\{i_{j} < \dots < i_{1} < i_{+}\}} \circ f_{\{i_{-} < i_{m} < \dots < i_{j}\}}
		\right)
		\]
	\end{enumerate}
	If $\alpha: [m] \to [n]$ is a non-decreasing function then the induced map $\Ndg(\D)_{n}\to \Ndg(\D)_{m}$ is given by 
	\[
	(\{X_{i}\}_{0 \leq i \leq n}, \{f_{I}\}) \mapsto (\{X_{\alpha(i)}\}_{0 \leq j \leq m}, \{g_{J}\})
	\]
	where for any $J \subset [m]$ we define
	$g_{J} = f_{\alpha(J)}$ if $\alpha|_{J}$ is injective, $g_{J} = \id_{X_{i}}$ if $J = \{j, j'\}$ and $\alpha(j)= \alpha(j')=i$ and $g_{J}=0$ otherwise.
	
\end{defi}
\begin{rem}
	Note that in \cite{Lurie11} it was shown that $\Ndg$ is right Quillen, but the left adjoint $L$ was not constructed explicitly. Explicit combinatorial constructions of the adjoint were given in \cite{Rivera19,Zeinalia16} and it was proved in \cite[Proposition 7.1]{Zeinalia16} that in 
	the one-object case it is given as the cobar-construction of the normalized chain coalgebra as in Proposition \ref{prop:dgagnerve} below.
\end{rem}

\subsection{The dg nerve of an algebra}
As a warm-up, we express the dg nerve of a dg algebra in terms of its bar construction.
We will find that it is given by a simplicial Maurer-Cartan set.
Let $\ground$ be a field for now.

We introduce the cosimplicial coalgebra $n \mapsto C_{*}(\Delta^{n},  \ground)$ given by the normalized chains on the standard $n$-simplex.
The coalgebras $C_{*}(\Delta^{n})$ are not conilpotent, but we may form a conilpotent quotient by identifying all the grouplike elements. More generally, given an arbitrary dg coalgebra $C$, form a noncounital coalgebra $C/C_0$ by quotienting out its coradical $C_0$, and add a one-dimensional space spanned by a group-like element (to make it counital). The resulting coalgebra $Q(C)$ will be conilpotent. The functor $Q$ is easily seen to be left adjoint to the inclusion from $\dgco$ into all dg coalgebras. 
We write $\Hom(\Delta, -)$ for the functor $\Hom_{\cuco}(QC_{*}(\Delta^{\bullet}), -)$ whose faces and degeneracies are induced by the corresponding maps on simplices. 

\begin{lem}\label{lem-chainsquillen}
	There is a Quillen adjunction $C_{*}: \qCat^{0} \rightleftarrows \cuco: \Hom(\Delta, -)$. 
\end{lem}
\begin{proof}
	As we may write the usual reduced chain coalgebra as $C_{*}(K) = \colim_{\Delta K} C_{*}(\Delta^{n})$ we have an adjunction $Q \circ C_{*}: \sSet \rightleftarrows \dgco: \Hom(QC_{*}(\Delta^{\bullet}), -)$.
	This factors through reduced simplicial sets $C_{*}: \sSet^{0} \rightleftarrows \dgco: \Hom(QC_{*}(\Delta^{\bullet}), -)$.
	
	By Lemma 3.4 in \cite{HolsteinH}
	$C_{*}$ is a left Quillen functor $\qCat^{0} \to \dgco$, which we compose with the left Quillen functor $i: \dgco \to \cuco$ from Corollary \ref{cor:uncurvingcoalg} and   Remark \ref{rem:hquillen}.
\end{proof}

\begin{defi}
	Let $A$ be a dg algebra. Then we  define its dg nerve $\Ndgp(A)$ as the 
	composition $\Hom(\Delta, \Ba(A))$.
\end{defi}
Using Lemma \ref{lem:conilpotentadjoint} we may write $\Ndgp(A)$ 
explicitly as the the simplicial set 
\[
n \mapsto  \MC(QC_{*}(\Delta^{n}), A) \cong \MC(C^*(\Delta^n, \tau_{\leq 0} A)).
\]
Here the truncation is needed as the MC construction with two arguments is $\MC(\Hom(QC_{<0}(\Delta^{n}), A))$.

Thus $\Ndgp(A)$ is the \emph{simplicial MC set} of the associative algebra $\tau_{\leq 0}A$. 
This construction of the simplicial MC set is analogous to (but different from) the simplicial MC set of a dg Lie algebra as constructed in \cite{Getzler09}, which uses differential forms on the $n$-simplex.

Then we have the following result.	
\begin{prop}\label{prop:dgagnerve}
	The functor $\Ndgp$ is a right Quillen functor from the category of  dg algebras to the category of reduced simplicial sets with the Joyal model structure. 
	Its left adjoint assigns to a reduced simplicial set $S$ the dg algebra $L(S):=\Omega(C_*(S))$, the cobar-construction of the chain coalgebra of $S$.
\end{prop}
\begin{proof}
	This is a direct consequence of Theorem \ref{thm-barcobaralgebra} and Lemma \ref{lem-chainsquillen}.
\end{proof}

\begin{prop}\label{prop:dgnerveone}
	$\Ndgp$ as defined above is equivalent to $\Ndg$ as in Definition \ref{def:lurienerve}.
\end{prop}
As we will work out the details in a more general case in Theorem \ref{thm:dgnerveadjoint}, we only give an outline of the proof of this result. 
\begin{proof}[Sketch of proof]
	The key observation is that we may define a curved coalgebra $\utilde C_{*}(\Delta^n)$ for the $n$-simplex $\Delta^n$ as the chain coalgebra whose differential is obtained by removing the boundary summands, i.e.\ $\utilde d_{n} = \sum_{i=1}^{n-1} (-1)^{i} \partial_{i} = \partial_{n} - \partial_{0} - (-1)^{n}\partial_{n}$.
	
	As curved coalgebras $C_{*}(\Delta^n)$ and $\utilde C_{*}(\Delta^n)$ are isomorphic via the isomorphism $(\id, e)$ where $e: C_{1}(K) \to \ground$ is constant and takes the value 1.
	
	Then unravelling Lurie's definition we have
	\begin{align*}
	\Ndg(A)_{n} &= \MC(\utilde C_{*}(\Delta^{n}), A) \\
	& = \Hom_{\cuco}(\utilde C_{*}(\Delta^{n}), \Ba A)  \\
	& = \Hom_{\cuco}(C_{*}(\Delta^{n}), \Ba A) = \Ndgp(A)_{n}.
	\end{align*}
	Compatibility with face and degeneracy maps can be checked.
\end{proof}

\subsection{The dg nerve revisited}
In this section we will define a dg nerve functor $\Ndgp:\dgCat\to\Sset$ drawing on our work from the previous section.
We then show our construction agrees with the explicit description in Definition \ref{def:lurienerve}. 

We would like to apply the adjunction from Proposition \ref{prop:catadjunction} in the case where $C$ is the normalised chain coalgebra of a simplicial set with coefficients in $\mathbb Z$. 
The underlying graded $\mathbb Z$-module of $C_*(K)$ is free, thus we may consider its cobar construction, compare the discussion before Corollary \ref{cor:koszulz}.

Note however, that $C_*$ is a pointed coalgebra with a dg structure and a unique splitting given by the projection $C_* \to C_0$, but $C_*$ is not a pointed curved coalgebra since this splitting is not compatible with the differential.
However, it is isomorphic to a pointed curved coalgebra.
We will describe this pointed model explicitly.
\begin{defi}
Given a simplicial set $K$ we define its \emph{twisted cochain algebra} $\utilde C^*(K)$ as the usual normalized cochain algebra equipped with the differential $\tilde \delta = \delta + [-e, ]$, i.e.\ $\utilde \delta f = \delta f - e \cup f + (-1)^{|f|} f \cup e$
where $e$ is the constant 1-cochain with value 1. The curvature is $\delta e + e \cup e$.
We note that $\tilde C^*(K)$ is pseudocompact since $C^*(K)$ is and the \emph{twisted chain coalgebra} $\utilde C_*(K)$ is defined as the continuous dual of $\utilde C^*(K)$. 
\end{defi}

\begin{example}\label{eg:differentialsimplex}
	Given the $n$-simplex $\Delta^{n}$ we find that $\utilde C_*(\Delta^{n})$ as a subspace of the unnormalized chains has a canonical basis given by subsets of $\{0, \dots, n\}$.
	Unravelling the definition the differential becomes $\utilde \partial = \sum_{i=1}^{k-1} (-1)^{i}\partial_{i}$. 
\end{example}
\begin{rem}\label{rem:boundarydifferential}
	More generally, we want to think of the twisted differential $\tilde \delta$ as removing the boundary terms.
	
	Indeed, if $\sigma$ is a $0$-simplex such that $\sigma_{01}$ and $\sigma_{n-1,n}$ are non-degenerate then $(e\cup f) = f(\sigma_{1\dots n}) = f(\partial_0 \sigma)$
	and $\tilde \partial  = \sum_{i=1}^{k-1} (-1)^{i}\partial_{i}$.
	
	However, if $\sigma_{01}$ or $\sigma_{n-1,n}$ is degenerate, this formula no longer holds since the constant 1-cochain sends all non-degenerate simplices to 1, but of course all degenerate simplices must map to 0.
	
	Similarly $-e$ satisfies the Maurer-Cartan condition, unless there are nondegenerate simplices $\sigma$ with degenerate $\sigma_{01}$ or $\sigma_{n-1,n}$.
\end{rem}
\begin{lem}\label{lem:twistedchainfunctor}
	$\utilde C_*$ is a functor from simplicial sets to $\ptdco^*$. 
	As a functor to curved coalgebras it is naturally isomorphic to the normalized chain coalgebra functor.
\end{lem}
\begin{proof}
	As the differential $\utilde \partial$ in degree 1 is just the zero map it is immediate that the natural projection $\tilde C_* \to C_0$ is compatible with the differential and $\utilde C_*(K)$ is a pointed curved coalgebra.
	
	The isomorphism is given by $(\id, -e)$ on objects and any map $f: K \to L$ of simplicial sets gives a map $(\id, -e) \circ f_* \circ (\id, e)$ on twisted chains.
	
\end{proof}

As $-e$ is not in general a Maurer-Cartan
element the differential $\utilde \partial$ does not square to 0 and the twisted chain coalgebra has nonzero curvature in general, cf.\ Remark \ref{rem:boundarydifferential}.

\begin{lem}\label{lem:tildeadjunction}
There is an adjunction $\utilde C: \sSet \rightleftarrows \ptdco^*: \Hom(\utilde C_*(\Delta^{\bullet}), -)$. We denote the right adjoint by $F$.
\end{lem}
Explicitly $F(\cat D)_n  = \Hom({\utilde C}_{*}(\Delta^{n}), \D)$ and the maps $[m] \to [n]$ induce natural maps $F(\cat D)_n \to F(\cat D)_m$ making it into a simplicial set. 

\begin{proof}
Let $K$ be a simplicial set and $C$ a pointed curved coalgebra.
As in the untwisted chain case we have $\utilde C_*(K) = \colim_{\Delta^n \in \Delta K}(\utilde C_*(\Delta^n), C)$ where $\Delta K$ is the simplex category of $K$.
Then by a standard argument
\begin{align*}
\Hom(\utilde C_*(K), C) &= \lim_{\Delta K}\Hom(\utilde C_*(\Delta^{n}), C) = \lim_{\Delta K}(\Delta^{n}, \Hom(\utilde C_*(\Delta^{\bullet}), C)) \\
&= \Hom(K, \Hom(\utilde C_*(\Delta^{\bullet}), C)) \qedhere
\end{align*}
\end{proof}

For later use we note down an explicit description of the functoriality for $\utilde C_*$ on the standard simplices.

\begin{lem} \label{lem:twistedchainssimplex}
	For a morphism $\alpha: [m] \to [n]$ we have $\utilde C_*(\alpha) = (\alpha_{*}, x_{\alpha}): \utilde C_*(\Delta^m) \to \utilde C_*(\Delta^n)$ where $\alpha_{*}$ sends $\sigma: [k] \to [m]$ to $\alpha \circ \sigma: [k] \to [n]$ and $x_{\alpha}(\sigma)$ is $0$, unless $\sigma$ is a 1-simplex and $\alpha \circ \sigma$ is degenerate.
	In the latter case $\alpha\circ \sigma$ is degenerate of the form $sp$ where $p$ is the point in $\Delta^n$ corresponding to $x_{\alpha}(\sigma) \in \utilde C_*(\Delta^n)$. 
\end{lem}
\begin{proof}
	We check on twisted cochains. 
	We need to check that $(\id, -e) \circ (\alpha_{*}, 0) = (\alpha_{*}, x_{\alpha}) \circ (\id, -e)$. 
	This is true if and only if $x_{\alpha} - e = -e \circ \alpha$. 
	This holds by unravelling definitions: $e$ and $e\circ{\alpha}$ agree except on 1-simplices which become degenerate under $\alpha$, and these are exactly the only simplices on which $x_{\alpha}$ is nonzero.
\end{proof}

\begin{defi}
The dg nerve $\Ndgp: \dgCat \to \qCat$ is defined as $\D \mapsto \MC({\utilde C}{}_{*}(\Delta^{\bullet}), \D)$. 
\end{defi}

\begin{rem}\label{rem:ndgpbar}
	If $\ground$ is a field we may write $\Ndgp(\cat D) = F\Ba(\cat D)$.
	This is immediate from Proposition \ref{prop:catadjunction} together with the isomorphism from Lemma \ref{lem:twistedchainfunctor}.
\end{rem}

\begin{lem}\label{lem:leftadjoint}
	The functor $\Ndgp$ is left adjoint to $\Omega \circ \utilde C_*$
\end{lem}
\begin{proof}
If $\ground$ is a field we may just combine the adjunctions of Proposition \ref{prop:catadjunction} and Lemma \ref{lem:tildeadjunction}
to obtain an adjunction between $F \circ \Ba$ and $\Omega \circ \utilde C_*$.
 
 If $\ground$ is not a field we argue as follows.
 Let $K$ be an arbitrary simplicial set, then we have:
 \begin{align*}
 \Hom(K, \Ndgp(\cat D)) & \cong \Hom(K, \MC({\utilde C}_{*} \Delta^\bullet, \cat D)) \\
 & \cong \Hom(K, \Hom(\Omega \utilde C_*(\Delta ^\bullet), \cat D))\\
 & \cong \lim_{\Delta K} \Hom(\Delta^n, \Hom(\Omega \utilde C_*(\Delta^\bullet), \cat D)) \\
 & \cong \lim_{\Delta K} \Hom(\Omega \utilde C_*(\Delta^n), \cat D) \\
 & \cong \Hom(\Omega \utilde C_*(K), \cat D)
 \end{align*}
 To get to the second line we use the bijection $\Hom_{\dgCat}(\Omega C, \cat D) \cong \MC(C, \cat D)$
 for $C$ in $\ptdco_{\operatorname{fr}}$. 
 This is the second isomorphism of Theorem \ref{thm-semibarcobar}. 
 Our ground coalgebra $C_0$ is now a direct sum of copies of $\mathbb Z$, but the proof is unaffected by working over the integers as $\utilde C_*$ is free over $\mathbb Z$.
 
Later we use the fact that $\Omega \circ \utilde C_*$ commutes with colimits.
 This is clear for $\utilde C_*$ by Lemma \ref{lem:tildeadjunction} and for $\Omega$ we argue as follows.
 Let $\cat D$ be an arbitrary dg category and $C$ be a split curved coalgebra with free underlying $\ground$-module.
 Let $C = \colim_i C^i$. 
 As all morphism spaces decompose over the coproduct of maps $f_O: C_0 \to \ground[\Ob(\cat D)]$ we consider each summand in turn and fix a map $f_O$ (equivalently a collection of maps $C^i_0 \to \ground[\Ob(\cat D)]$).
Then we have
 $\Hom(\colim_i \Omega C^i, \cat D) \cong \lim_i \Hom(\Omega C^i, \cat D) \cong \lim_i \MC(\Hom(\bar C^i, \cat D)) 
 \cong \MC(\lim_i \Hom(\bar C^i, \cat D)) \cong \MC(\Hom( \colim_i \bar C^i, \cat D)) \cong \Hom(\Omega \colim_i C^i, \cat D)$. 
 Here we use the fact that taking Maurer-Cartan elements commutes with limits and taking the quotient by the coradical commutes with colimits in $\ptdco^*$.
\end{proof}
\begin{theorem}\label{thm:dgnerveadjoint}
In the adjunction $L: \qCat \rightleftarrows \dgCat: \Ndg$ described in \cite{Lurie11} there are natural isomorphisms $\Ndg(\cat D) \cong \Ndgp(\cat D)$ 
and $L(K) \cong \Omega \utilde C_*(S)$.
\end{theorem}

\begin{proof}
We will explicitly compare $\Ndgp(\cat D) \coloneqq \MC(\utilde C_{*}(\Delta^\bullet), \cat D)$ with Lurie's construction of $\Ndg$ as recalled in Definition  \ref{def:lurienerve}.

We recall first one difference in convention: We write our composition as $(f,g) \mapsto f\circ g$ while Lurie uses the convention $(f,g) \mapsto g\circ f$, cf.\ Remark \ref{rem-order}.

For $n=0$ we have $\Ndgp(\cat D)_{0} = \Ob(\cat D) = \Ndg(\cat D)$.
To be precise, unravelling the definitions we have 
$\Ndgp(\cat D)_{0} = \amalg_{D \in \Ob(\cat D)} \MC(0, \cat D)) = \amalg_{D \in \Ob(\cat D)} *$. 

For $n \geq 1$ we have
\begin{align*}
\Ndgp(\cat D)_n & \cong \coprod_{O: \Ob([n]) \to \Ob(\cat D)} 
\MC \left[\prod_{s, t\in\Ob([n])} \Hom_{\ground}\left(\ground .e_{s} \boxempty _R \utilde C_{<0}(\Delta^{n}) \boxempty_R \ground.e_{t}, \Hom_{\cat D}(O(s),O(t)\right) \right] \\
\\ 
& \cong \coprod_{O: \Ob([n]) \to \Ob(\C)} 
\MC \left[\prod_{s, t\in\Ob([n])} e_{s} \utilde C^{>0}(\Delta^{n}) e_{t} \otimes_{\ground} \Hom_{\cat D}(O(s),O(t)) \right].
\end{align*}
Here we write $R$ for the coalgebra ${\ground[\Ob([n])]}$.
The differential is induced by the differential in $\Hom_{\cat D}$ and the differential $\sum_{i=1}^{n-1}(-1)^{i}\partial_{i}$ in $\utilde C_*(\Delta^{n})$, see Example \ref{eg:differentialsimplex}. 
The product is induced by the coproduct $\sigma \mapsto \sum_{i=0}^{n} \sigma_{0 \dots i} \otimes \sigma_{i \dots n}$ where $\sigma_{i\dots j}$ is the restriction of an $n$-simplex along $[i,j] \subset [0,n]$.
(We could remove the $i=0$ and $i=n$ terms as they are necessarily 0 and do not contribute.)

Unravelling definitions, we conclude that this is exactly $\Ndg(\cat D)_n$ as defined by Lurie (once we flip the order of composition).

For the face and degeneracy maps, consider the rules for $\alpha: [m] \to [n]$ inducing a map $\Ndg(\cat D)_n \to \Ndg(\cat D)_{m}$ in Definition \ref{def:lurienerve} and compare with the functoriality of $\utilde C_*$ set out in Lemma \ref{lem:twistedchainfunctor}.

For any subset $J \subset [m]$ we distinguish three cases:

If $\alpha|_{J\subset [m]}$ is injective we just consider the map $f \mapsto f \circ \alpha|_{J}$ on cochains in both cases.

If $\alpha|_{J\subset [m]}$ is not injective then in general the induced map on $\Ndg$ is $0$.
This agrees with our definition on $\Ndgp$ because the image of $J$ is a reduced simplex and there is no MC element associated to it. 

However, there is an exception in the case that $\alpha|_{J}$ is of the form $[1] \to [0]$. 
In this case Definition \ref{def:lurienerve} sends $X_{0} \in \Ndg(\cat D)_{0}$ to the data $((X_{0}, X_{0}), \id_{X_{0}})$ in $\Ndg(\cat D)_{1}$.

By Lemma \ref{lem:twistedchainssimplex} the map on $\Ndgp(C)$ is induced by $(\alpha_*, x_\alpha)$. Let now $\alpha|_J: [1] \to [0]$. 
This is the only case when $x_\alpha$ takes a nonzero value. 
We split the semialgebra corresponding to $\cat D$ as $V_{\cat D} = \bar V_{\cat D} \oplus \ground[\Ob(\cat D)]$ where the second summand is identified with the units.
The MC element $\xi: \utilde C_0 \to \cat D$ is sent to $ \xi\circ \alpha + x_\alpha: \utilde C_1 \to \bar V_{\cat D} \oplus \ground[\Ob(\cat D)]$. 

The first summand vanishes as $\xi \circ \alpha$ is degenerate, and the second summand picks out the element $e_{X_0} \in \ground_{\Ob(\cat D)}$ corresponding to the identity at $X_0$. Thus the map induced by $\alpha$ sends $X_{0} \in \Ndgp(\cat D)_{0}$ to $((X_{0}, X_{0}), \id_{X_{0}})$ in $\Ndg(\cat D)_{1}$. 
\end{proof}

We have the following corollary of Theorem \ref{thm:dgnerveadjoint}, which makes precise that Koszul duality is a linear version of the coherent nerve construction, an idea which was previously exploited in \cite{HolsteinH}.

\begin{cor}\label{cor:nervekoszul}
	Consider the following diagram of $\infty$-categories
	
	\[	\xymatrix@=32pt{
		\qCat\ar@<0.5ex>[r]^{\mathfrak C}\ar[d]_{\utilde C_*}
		&\sCat\ar@<0.5ex>[l]^{\Ncoh}\ar[d]^{G_*}\\
		\ptdco^* \ar@<0.5ex>[r]^{\Omega}
		&\dgCat\ar@<0.5ex>[l]^{\Ba}
	}
	\]
	where the horizontal arrows are $\infty$-equivalences and downward arrows are induced by normalized chain functors. 
	Then in the associated diagram of homotopy categories
	we have natural equivalences $G_*\circ\mathfrak{C}\simeq \Omega\circ \utilde{C}_*$ and  $\utilde{C}_*\circ \Ncoh\simeq \Ba\circ \utilde{C}_*$.
	
	If $\ground$ is not a field we understand the categories on the bottom to be $\ptdco^*_{\fr}$ and $\dgCat_{\fr}$.
\end{cor}
\begin{proof}
	The right vertical functor $G_*: \sCat \to \dgCat$ is more precisely the composition $i \circ N_{DK} \circ \ground$ where on hom spaces $i$ is inclusion of non-negative complexes, $N_{DK}$ is normalization and $\ground$ is the free functor.
	There is a natural functor $H = U \circ DK \circ \tau_{\geq 0}: \dgCat \to \sCat$ which is a right adjoint to $G_*$ on the level of homotopy categories. Note that $ DK $ is a left adjoint, see \cite{Tabuada10}, 
	but it induces an equivalence of homotopy categories.
	
	The composition $\Ncoh \circ H$ is equivalent to the dg nerve $\Ndg$, see  \cite[Proposition 1.3.1.17]{Lurie11}.
	$\Ndg$ is right Quillen 
	and its left adjoint is $L$ by Theorem \ref{thm:dgnerveadjoint}.
	
	Thus on the level of homotopy categories the two left adjoints of $\Ndg$ must agree and $L(S) = \Omega \utilde C_{*}(S)$ is homotopy equivalent to $G_* \circ \mathfrak C$. 
	In the other direction we have
	\[\Ba \circ G_* \simeq \Ba \circ G_* \circ \mathfrak C \circ \Ncoh \simeq \Ba \circ \Omega \circ \utilde C_* \circ \Ncoh \simeq \utilde C_* \circ \Ncoh.\]
	Here we use the fact that the top and bottom row are weak equivalences, by Theorem \ref{thm:koszulquillen} resp.\
	Corollary \ref{cor:koszulz}.	
\end{proof}

\begin{cor}\label{cor:ctildequillen}
		The functor $\tilde C_*: \qCat \to \ptdco^*$ preserves weak equivalences.
		If $\ground$ is a field it is left Quillen.
\end{cor}
\begin{proof}
	Since $G_*$ and $\mathfrak C$ preserve weak equivalences so does $\tilde C_*$ by Corollary \ref{cor:nervekoszul}.
	If $\ground$ is a field we have model structures and it follows from the definitions that $\tilde C_*$ preserves cofibrations.	
\end{proof}

\begin{rem}
	It follows from Theorem \ref{thm:koszulquillen} and Corollary \ref{cor:ctildequillen} that $L \dashv \Ndg$ is Quillen if $\ground$ is a field.
	
	This result holds more generally, see \cite[Proposition 1.3.1.20]{Lurie11}. 	
\end{rem}

	
	
\begin{cor}\label{cor:chainalgebra}
	For any grouplike simplicial set $K$, the dg category $\Omega \tilde{C} K$ is naturally quasi-equivalent to the dg category obtained by applying the normalized singular chains functor to the topological path category of $|K|$, the geometric realization of $K$.
\end{cor}
\begin{proof}
Let us first assume that $K$ is connected. Then $K$ is categorically equivalent to a reduced group-like simplicial set.
Corollary \ref{cor:ctildequillen} and  Theorem \ref{thm:dgnerveadjoint} show that categorical equivalences are preserved by $\Omega \tilde C_*$.
The result now follows from \cite[Corollary 4.2]{HolsteinH}.

	The general statement now follows since $\Omega \tilde C_*$ preserves coproducts (i.e. disjoint unions) of simplicial sets. 
\end{proof}	
\begin{rem}
	Corollary \ref{cor:chainalgebra} generalizes \cite[Corollary 4.2]{HolsteinH} from reduced to connected (grouplike) simplicial sets.
	A similar argument generalizes \cite[Corollary 4.7]{HolsteinH} to connected Kan complexes.
\end{rem}

\begin{rem}\label{rem:mistake}
	We had originally anounced that $\utilde{C}_*: \qCat^0 \to \ptdco^*$ reflects weak equivalences (Corollary 4.22 in the previous version of this paper).
	Our proof was flawed and George Raptis pointed out the following counterexample: 
	Recall that if $\tilde C_*$ reflects weak equivalences then so does $\Omega \tilde C_*$, and thus $G_*\mathfrak C$. 
	Thus for a counterexample it suffices to present two simplical monoids which are homology isomorphic but not weakly equivalent.
	Let $A$ be some connected acyclic simplicial set and consider the free simplicial monoid on $A$. It will be homology isomorphic to the discrete monoid of the natural numbers, but not weakly equivalent.
	
	In other words a reduced quasi-category with endomorphisms freely generated by $A$ and a reduced quasi-category with endomorphisms freely generated by one morphism are $\tilde C_*$-equilent, but not categorically equivalent. 
\end{rem}

\section{Functor categories and comodules}\label{sect:functorcategories}
\begin{lem}\label{lem:quasidg}
Let $S$ be a simplicial set.
We denote by $R\Hom$ the derived internal hom of dg categories.
Then there is a categorical equivalence of quasicategories 
\[
\Ndg(R\Hom(L(S), \Ch)) \simeq \Fun(S, \Ndg(\Ch)).\]
In particular this implies that $\Fun(S, 
\Ndg(\Ch))$ is categorically equivalent to $\mathscr D(L(S))$, the $\infty$-category of dg modules over $L(S)$.
\end{lem}
\begin{proof}
By the Yoneda lemma it suffices to compare the functors $\Hom(-, \Ndg(\Fun(L(S), \Ch)))$ and $\Hom(-, \Fun(S, \Ndg(\Ch)))$ on the homotopy category of quasicategories. We recall that $L \dashv \Ndg$ induces an adjunction of homotopy categories between quasicategories and dg categories by \cite[Proposition 1.3.1.20]{Lurie11}.

Let $K$ be any simplicial set. By the proof of  \cite[Proposition 1.2.7.3]{Lurie11a} we have 
\begin{align*}
\Hom_{\operatorname{Ho}(\qCat)}(K, \Fun(S, \Ndg(\Ch))) &\simeq \Hom_{\operatorname{Ho}(\qCat)}(K \times S, \Ndg(\Ch))\\
& \simeq  \Hom_{\operatorname{Ho}(\dgCat)}(L (K \times S), \Ch) \\
& \simeq \Hom_{\operatorname{Ho}(\dgCat)}(LK \otimes LS, \Ch) \\
& \simeq \Hom_{\operatorname{Ho}(\dgCat)}(LK, R\Hom(LS, \Ch)) \\
& \simeq \Hom_{\operatorname{Ho}(\qCat)}(K, \Ndg(R\Hom(LS, \Ch))
\end{align*}
To show that $L$ sends products to tensor product in the homotopy category of dg categories we used that as in Corollary \ref{cor:nervekoszul} we may write $L = G_* \circ \mathfrak C$ on the level of homotopy categories.
Then $\mathfrak C$ preserves products by  \cite[Corollary 2.2.5.6]{Lurie11a}. 
Note that $LS$ is cofibrant so we may consider the tensor product as derived.
The derived tensor product in $\dgCat$ is adjoint to the derived internal hom by \cite[Corollary 6.4]{Toen06}.
 
For the final part we recall from \cite{Toen06} that $R\Hom(LS, \Ch)$ has an explicit model given by the dg category of fibrant cofibrant right quasi-representable $LS  \otimes^L \Ch^{op}$-modules. 
Unravelling the definitions, we see that these are exactly fibrant cofibrant dg modules over $LS$. 
The dg nerve of the subcategory of fibrant cofibrant objects in a dg model structure is equivalent to the localization at weak equivalences, thus $\Ndg(R\Hom(L(S), \Ch)) \simeq \mathscr D(L(S))$. 
\end{proof} 

Note that the quasicategory $\Fun(S, \Ndg(\Ch))$ can be viewed as the derived $\infty$-category of the quasicategory $S$.

\begin{theorem}\label{thm:main}
Let $S$ be a simplicial set considered as a quasicategory. 
Then there is an equivalence of $\infty$-categories between
$\Fun(S, \Ndg(\Ch))$ and $\mathscr D^{\co}(C_*S)$. 
\end{theorem}
\begin{proof}
	Theorem \ref{thm:modulecomodule} (with the notation of Definition  \ref{defi:derivedinfinitysecond}) provides an equivalence $\Dco(\utilde C_*S) \simeq \mathscr D(\Omega(\utilde C_*S))$, where the right hand side is just the $\infty$-categorical localization of $\Omega(C_*S)$-modules at quasi-isomorphisms.
	The isomorphism $C_*S \cong \utilde C_*S$
	gives $\mathscr D^{\co}(\utilde C_*S) \simeq \mathscr D^{\co}(C_*S)$.
	
	By Theorem \ref{thm:dgnerveadjoint} we know $\Omega(\utilde C_*S) \simeq L(S)$ and thus the result follows from Lemma \ref{lem:quasidg}.
\end{proof}
\begin{rem}
	The coderived $\infty$-category of $C_*(S)$ is also equivalent to (the $\infty$-version of) the category of twisted modules over $C^*(S)$ as they were considered in \cite{HolsteinG}.
\end{rem}
We also have the following consequence of our earlier results.
\begin{prop}
	For any simplicial set $S$, the Morita fibrant replacement of the dg category $L(S)^{\op}$ is quasi-equivalent to the category of finite dimensional  $C_*(S)$ comodules.
	This may be viewed as the perfect derived category of $S$. 
\end{prop}
\begin{proof}
	Note that we can identify the dg category generated by one-dimensional comodules with the image of $L(S)^{\op}$ under the Yoneda embedding  in $L(S)\Mod$. Indeed,
	for any vertex $s \in S_0$ consider the corresponding one-dimensional comodule $\ground_s$ of $C_* S$ (all 1-dimensional  $C_* S$-comodules are of this form).
	By Example \ref{eg:representables} the adjunction 
	from Proposition \ref{prop:fgadjoint} identifies $\ground_s$ with the right $L(S)$-module represented by $s \in \Ob(L(S))$. 
	
	To complete the proof of the corollary we close both sides of this correspondence under extensions and suspensions.
\end{proof}
\begin{cor}
	The functor $S \mapsto \Dco(C_{*}(S))$ sends colimits to limits.
\end{cor}
\begin{proof}
This is immediate from Theorem \ref{thm:main}.
\end{proof}

	\section{Stratified spaces}	\label{sect:stratified}
	We finish with an application of Theorem \ref{thm:main} to stratified spaces.
	 We will use the terminology of \cite[Appendix B]{Lurie11}. 	
	Let $X$ be a paracompact topological space which is locally of singular shape and is equipped with a conical $A$-stratification where $A$ is a partially ordered set satisfying the ascending chain condition. 
	We write $\operatorname{Exit}(X)$ for the $\infty$-category of exit paths of the $A$-stratified space $X$, denoted by $\operatorname{Sing}^A(X)$ in \cite{Lurie11}.
	
	Then Theorem A.9.3 of \cite{Lurie11} (together with the discussion just  before Construction A.9.2) states that the $\infty$-categories of constructible sheaves of spaces on $X$, written as $\operatorname{Constr}(X, \cat S)$ is equivalent to $\Fun(\operatorname{Exit}(X), \cat S)$ where $\mathcal S$ is the quasicategory of spaces. We write this 
	 equivalence as $\Psi$.
	
	We are interested in the linear version. This is probably well-known to experts, we outline a proof for lack of a reference.
	\begin{prop}\label{prop:lurie}
		The $\infty$-category of constructible sheaves of cochain complexes on $X$ is equivalent to $\Fun(\operatorname{Exit}(X),\Ndg(\Ch))$.
	\end{prop}
\begin{proof}
	We first extend the result to spectra by considering the associated stabilizations. 
	We may for example identify categories of spectrum objects in the sense of  \cite[Section 1.4.2]{Lurie11} for both categories, obtaining an equivalence
	$\Psi: \Sp(\Fun(\operatorname{Exit}(X), \cat S)) \simeq \Sp(\operatorname{Constr}(X, \cat S))$.
	
	By \cite[Remark 1.4.2.]{Lurie11} we have $\Sp(\Fun(\operatorname{Exit}(X), \cat S)) \simeq \Fun(\operatorname{Exit}(X), \Sp(\cat S))$, and of course $\Sp(\cat S)$ is the $\infty$-category of spectra $\Sp$.
	
	It remains to check that spectrum objects in constructible sheaves of spaces are constructible sheaves of spectra.
	
	The stabilization of the $\infty$-category of sheaves of spaces is the $\infty$-category of
	sheaves of spectra, see  \cite[Remark 1.2]{Lurie11c}.
	Moreover, the subcategories of constructible sheaves can be identified. 
	Constructible sheaves are those whose pullback to small enough open subsets of strata of $X$ are  constant.
	We first observe that constant sheaves on both sides are identified. 
	Thus it suffices to show that the equivalence of sheaves of spectra with the stabilization of sheaves of spaces is compatible with pullbacks. 
	We consider the stabilization as the homotopy limit of a tower of loop functors, see  \cite[Proposition 1.4.2.24]{Lurie11}.
	As the pullback functor is left exact it commutes with constructing the loop functor and we can identify constructible sheaves as desired.

	Next we replace spectra by (unbounded) chain complexes.
	By the stable Dold-Kan correspondence $\Ch_{\ground}$ is equivalent to $\ground$-modules in spectra, and these may be characterized as $\ground$-module objects in $\Sp$, i.e.\ abelian group objects in the cartesian monoidal category $\Sp$ equipped with a compatible action of the monoid $\ground$. 
	
	We now claim that the $\infty$-categories of the corollary arise as $\ground$-module objects in the categories considered above.
	To be precise, the equivalence $\Psi$, and the induced map on constructible spectra, is an equivalence of cartesian monoidal categories and it identifies the constant functor $\ground$ with the constant sheaf with value $\ground$. 
	We write $\ground$ for both.
	In both categories $\ground$ is a monoid and we may identify the $\ground$-module objects of $\operatorname{Constr}(X, \Sp)$ and $\Fun(\operatorname{Exit}(X), \Sp)$.
	
	These are the objects we are interested in, as the product of sheaves or functors is defined object-wise.
\end{proof}
	
We may now interpret constructible sheaves as follows:
	
	\begin{prop}\label{prop:constructiblecomodules}
		Let $X$ be a topological space with an $A$-stratification as above.
		Then the derived $\infty$-category of constructible sheaves of chain complexes on $X$ is equivalent to $\Dco( C_*\operatorname{Exit}(X))$.
	\end{prop}
\begin{proof}
	This follows by combining Theorem \ref{thm:main} with Proposition \ref{prop:lurie}.
\end{proof}
	
	Of course the exit path category is in general quite unwieldy.
	However, it forms part of an adjunction with the left adjoint given by a \emph{stratified realization functor} $K \mapsto ||K||$ from simplicial sets to stratified spaces, see  \cite[Definition 7.1.0.1]{Nand-Lal19}.
	To construct the stratified realization we send $\Delta^{n}$ to  $|\Delta^{n}|$ stratified by the $k$-simplices spanned by the first $k+1$ vertices for all $k \geq 0$, and then extend by colimits.
	\begin{cor}
		Let $K$ be a simplicial set. If $K \simeq \operatorname{Exit}(||K||)$ then the $\infty$-category of constructible sheaves on the stratified space $||K||$ is categorically equivalent to $\mathscr D^{\co}(C_*K)$.
	\end{cor}
	While $K \simeq \operatorname{Exit}(||K||)$ does not hold for all simplicial sets, there is a natural class for which one can expect it.
	\begin{conj}
		Let $K$ be a quasicategory in which all endomorphisms are equivalences.
		Then the $\infty$-category of constructible sheaves on the stratified space $||K||$ is categorically equivalent to $\mathscr D^{co}(C_*K)$.
	\end{conj}
	This is in line with Conjecture 0.0.8 in \cite{Ayala15} which states that these are exactly the quasi-categories arising as exit paths of stratified spaces.

\bibliography{./biblibrary2}

\end{document}